\documentclass[11pt]{article}
\usepackage{amsfonts,amsmath,amssymb,amsthm}
\usepackage{bm,hyperref}
\usepackage[normalem]{ulem}
\usepackage{graphicx} 
\usepackage[margin=1in]{geometry}
\usepackage{enumerate}
\usepackage{tikz-cd}
\usepackage{xcolor}
\usepackage{comment}
\usepackage{float}

\newtheorem{theorem}{Theorem}
\newtheorem{proposition}{Proposition}

\newtheorem{lemma}{Lemma}
\newtheorem{remark}{Remark}
\newtheorem{definition}{Definition}
\newtheorem{assumption}{Assumption}

\DeclareMathOperator{\curl}{curl}
\DeclareMathOperator{\divv}{div}

\newcommand{\RT}{\mathcal{RT}}

\usepackage[textsize=tiny]{todonotes}
\title{Structure-Preserving Optimal Control of Maxwell’s Equations with Applications to Source Cloaking}
\author{Harbir Antil$^{1}$, Yaw Owusu-Agyemang$^{1}$, Rohit Khandelwal$^{1}$, \\ Jimmie Adriazola$^{2}$, Denis Ridzal$^{3}$ \\
$^{1}$Center for Mathematics and Artificial Intelligence and \\ Department of Mathematical Sciences, George Mason University, \\ Fairfax, VA 22030. \\
$^{2}$School of Mathematical and Statistical Sciences,  Arizona State University, \\ Tempe, AZ 85281. \\
$^{3}$Optimization and Uncertainty Quantification, Sandia National
	Laboratories, \\ PO Box 5800, Albuquerque, NM 87185-1320.
\thanks{
This research was sponsored, in part by, Sandia National Laboratories' LDRD project Active Circuits for EMI
Resilience in Contested Environments (ACEM). Sandia
National Laboratories is a multimission laboratory managed and operated by
National Technology and Engineering Solutions of Sandia, LLC., a wholly owned
subsidiary of Honeywell International, Inc., for the U.S.~Department of
Energy's National Nuclear Security Administration under contract DE-NA0003525.
This paper describes objective technical results and analysis. Any subjective
views or opinions that might be expressed in the paper do not necessarily
represent the views of the U.S.~Department of Energy or the United States
Government. \\
This work is also partially supported by the Office of Naval Research (ONR) under Award  NO: N00014-24-1-2147, NSF grant DMS-2408877, NSF grant PHY-2316622, and the Air Force Office of Scientific Research (AFOSR) under Award NO: FA9550-25-1-0231. 
}
}
\date{}

\begin{document}
\maketitle

\begin{abstract}
We develop a structure-preserving solution framework for the optimal control of the time-dependent Maxwell's equations.
Building on a well-posedness theory for a weak form of the forward problem, we first analyze a forward solver that couples N\'ed\'elec and Raviart--Thomas finite elements with Crank--Nicolson time stepping.
The solver preserves the de~Rham structure, enforces a discrete Gauss law, exactly satisfies a per-time-step energy balance, and converges to the weak solution under low regularity assumptions on the problem data, which are dictated by the optimal control setting. 
To control the Maxwell system, we add the curl of a space-time current density as a source to Amp\`ere's law.
The curl form yields charge conservation without auxiliary constraints.
We prove the well-posedness and continuity of the control-to-state map, derive the adjoint system and a gradient representation for a tracking-type objective functional, and formulate a discrete optimization scheme that inherits structure preservation from the forward solver.
Our discrete stationarity conditions are consistent with their continuous counterparts, and the discrete optimal controls converge, with mesh and time refinements, to the continuous optima. 
We demonstrate the merits of our optimal control formulation and the theoretical developments by numerically solving a series of source-cloaking model problems.
\end{abstract}

\section{Introduction}

Maxwell’s equations are foundational to a variety of engineering applications, including nanoscale optics, electromagnetic shielding, non-destructive evaluation, medical imaging, and magnetic confinement fusion.
Such applications routinely feature nonsmooth geometries, discontinuous material tensors, and sources of only square-integrable regularity in time and space.
In this regime, it is essential to preserve the geometric structure of Maxwell's equations---tangential boundary conditions, Gauss's laws, and the curl–div sequence---both for the forward numerical simulation and for the numerical optimization motivated by optimal control and optimal design applications.

In this paper, we develop a structure-preserving solution framework for the optimal control problem
\begin{equation}\label{eq:intro_obj}
\min_{z} 
\left\{ J(E,B,z) := \frac12 \int_0^T {\left(\| \epsilon^\frac12 (E-E_d)\|^2_{L^2(\Omega_{\text{obs}})^3} + \| \mu^{-\frac12} (B-B_d)\|^2_{L^2(\Omega_{\text{obs}})^3} \right)} dt+\frac12 \|z\|^2_Z 
\right\}
\end{equation}
subject to Maxwell's equations on a bounded Lipschitz domain $\Omega\subset\mathbb{R}^3$ over a finite time horizon~$(0,T)$,
\begin{equation}\label{eq:state-ocp}
\left\{
\begin{aligned}
&\varepsilon \partial_t E - \mathrm{curl}\big(\mu^{-1} B\big) + \sigma E 
   = \chi_{\Omega_{\text{src}}} I_{\rm src} + \chi_{\Omega_{\text{ctrl}}} \, \mathrm{curl}\, z 
   &&\text{in }\Omega\times(0,T),\\
&\partial_t B + \mathrm{curl}\,E = 0 
   &&\text{in }\Omega\times(0,T),\\
&E\times\nu = 0 &&\text{on }\partial\Omega\times(0,T),\\
&E(0)=E_0,\quad B(0)=B_0 &&\text{in }\Omega,
\end{aligned}
\right.
\end{equation}
including Gauss's law for electricity $\partial_t \divv(\varepsilon E) + \divv(\sigma E-\chi_{\Omega_{\text{src}}} I_{\rm src} )=0$ and Gauss's law for magnetism $\divv B(\cdot,t) = 0$. 
Here $E$ is the electric field, $B$ is the magnetic flux density, and $\curl z$ is the unknown control current density.
Given a source current density~$I_{\rm src}$ confined to the region~$\Omega_{\text{src}} \subseteq \Omega$, the role of the control~variable~$z$, defined in the region~$\Omega_{\text{ctrl}} \subseteq \Omega$, is to steer~$E$ and~$B$ toward the given target functions~$E_d$ and~$B_d$, respectively, in the observation region~$\Omega_{\text{obs}} \subseteq \Omega$.
Moreover, $\varepsilon$, $\mu$, and $\sigma$ are the electric permittivity, magnetic permeability, and electric conductivity tensors, respectively,
while~$E_0$ and~$B_0$ are the initial data.
The Maxwell system is augmented by the perfectly electrically conducting (PEC) boundary condition~$E\times\nu=0$ on~$\partial\Omega$, where~$\nu$ is the outward-pointing unit vector normal to $\partial \Omega$.
To ensure conservation of the control charge density, we consider the control current density $\curl z$ instead of $z$.
Indeed, recall that $\divv \curl z = 0$.

Building on our theory for the forward problem~\cite{HAntil_2025a}, we work with uniformly elliptic $\varepsilon,\mu\in L^\infty(\Omega)^{3\times 3}$ and nonnegative $\sigma\in L^\infty(\Omega)^{3\times 3}$.
For the source $I_\text{src} \in L^2\big(0,T;L^2(\Omega_\text{src})^3\big)$ and initial data $(E_0, B_0) \in L^2\big(0,T;L^2(\Omega)^3\big)\times L^2\big(0,T;L^2(\Omega)^3\big)$, the weak solution class is
\[
E \in H^1\!\big(0,T;H_0(\curl;\Omega)^*\big)\cap C([0,T];L^2(\Omega)^3),\qquad
B \in H^1\!\big(0,T;H(\curl;\Omega)^*\big)\cap C([0,T];L^2(\Omega)^3),
\]
with the PEC boundary condition $E\times\nu=0$ on $\partial\Omega$.
As in~\cite{HAntil_2025a}, to discretize $E$ and $B$ in space, we use the de~Rham-compatible N\'ed\'elec and Raviart-Thomas (RT) spaces $\mathcal N_h^0\subset H_0(\curl;\Omega)$ and $\mathcal{RT}_h\subset H(\divv;\Omega)$, respectively, with $\curl\,\mathcal N_h^0\subset\mathcal{RT}_h$, which enforces a discrete Gauss law and supports energy identities at the semi-discrete level.
We use the spaces
\[
Z := L^2\big(0,T; H(\mathrm{curl};\Omega_\text{ctrl})\big), 
\qquad 
Y := L^2\big(0,T;L^2(\Omega_\text{obs})^3\big)\times L^2\big(0,T;L^2(\Omega_\text{obs})^3\big),
\]
for the control $z$ and the target data $E_d$ and $B_d$, respectively.
For $\alpha_1 > 0$ and $\alpha_2 > 0$, the control norm is given by
\[
\|z\|_Z^2 := \int_0^T \big( \alpha_1\|z(t)\|_{L^2(\Omega_\text{ctrl})^3}^2 + \alpha_2\|\mathrm{curl}\,z(t)\|_{L^2(\Omega_\text{ctrl})^3}^2 \big)\,dt.
\]

\paragraph{Main contributions.}
Our contributions are six-fold.
First, we extend the analysis of semi-discrete convergence from \cite{HAntil_2025a} to develop a structure-preserving \emph{fully discrete} forward scheme based on Crank–Nicolson (CN) time stepping, coupled with the N\'ed\'elec/RT finite element discretization in space.
We prove a per-step energy balance, stability for nonnegative $\sigma$, preservation of a discrete Gauss law for the magnetic induction, and convergence of fully discrete solutions to the unique weak solution under low-regularity assumptions on problem data on Lipschitz domains.
Second, we show that for every control $z\in Z$, the Maxwell system admits a unique weak solution $(E,B)$ in the energy class.
The control-to-state map $S:z\mapsto(E,B)$ is affine and bounded, yielding existence of optimal controls for convex quadratic objective functionals and closed convex admissible sets.
Third, we derive the adjoint Maxwell system in dual energy spaces and prove that the reduced objective functional is G\^ateaux differentiable in $z$, with the gradient represented via the adjoint variables.
This yields necessary and, with convexity assumptions, sufficient optimality conditions.
Moreover, the adjoint derivation is combined with the discrete CN-N\'ed\'elec/RT scheme used for the forward problem, resulting in a structure-preserving discrete adjoint scheme.
Fourth, using commuting projections and the de~Rham complex, we obtain a consistent discrete gradient of the reduced objective functional.
In other words, the discrete stationarity conditions match the Galerkin optimality system, enabling a structure-preserving discrete optimization scheme.
Fifth, we prove convergence of discrete optimal controls as the mesh size $h$ and the time step $\Delta t$ tend to zero.
Specifically, the discrete optimal triplets $(E_{h,\Delta t},B_{h,\Delta t},z_{h,\Delta t})$ admit subsequences converging to a continuous optimal triplet $(E,B,z)$; with strict convexity in $z$, the full sequence converges.
Sixth, using the developed structure-preserving scheme, we solve a series of source-cloaking model problems (with $E_d=0$ and $B_d=0$), in one, two, and three dimensions.

\paragraph{Relation to prior work.}
Our analysis sits at the intersection of structure-preserving finite elements (N\'ed\'elec/RT pairs and the de~Rham sequence \cite{monk2003finite, Hiptmair2002, bossavit1998, phillips-shadid2018maxwell}), energy-stable time stepping (e.g., Crank--Nicolson \cite{cranknicolson1947, vidarthomee2006, MortonMayersNumerical2005, strikwerda2004}), and the theory of Maxwell's equations with low regularity assumptions \cite{DuvautLions1976, HAntil_2025a}. Relative to this literature, the present work advances the state of the art in two major directions: 
(i) we prove convergence of a \emph{fully} discrete, structure-preserving scheme to weak solutions assuming $L^2$ data on Lipschitz domains; and 
(ii) we develop a convergent structure-preserving optimal control framework that operates entirely in this low-regularity setting. 
To ensure charge conservation without additional constraints, the control enters Amp\`ere’s law through a curl lift, $\mathcal{C} z = \curl z$, for $z\in L^2(0,T;H(\curl;\Omega))$, so that $\divv(\curl z)=0$ holds identically.
If charge conservation is not required, the control space can be relaxed to $L^2(0,T; L^2(\Omega)^3)$ with $\mathcal C=\mathcal I$.

To our knowledge, the only work that treats optimal control of the \emph{fully time-dependent} Maxwell system is \cite{bommer2016optimal}.
We discuss four methodological differences. 
First, our tracking objective functional is formulated over the full time horizon $(0,T)$, enabling, e.g., cloaking for all times, whereas \cite{bommer2016optimal} considers a terminal-time target. 
Second, we do not use a space-time separable control $z(x,t) = u(x)a(t$); in fact, the example in Figure~\ref{fig:split_vs_nosplit} demonstrates that, even in a simple one-dimensional setting, a separable control cannot effectively cloak typical sources $I_{\text{src}}(x,t)$, while a non-separable control $z(x,t)$ can. 
Third, in~\cite{bommer2016optimal} charge conservation is enforced via an explicit divergence-free constraint on the control, including additional time regularity, while in our formulation it is built in through the curl lift and does not require additional smoothness. 
Finally, in the context of the optimal control problem with low data regularity, we provide a convergence analysis for the widely accepted CN-N\'ed\'elec/RT discretization, which is not directly supported by~\cite{bommer2016optimal}.

 \begin{figure}[h!]
    \centering
    \includegraphics[width=0.48\linewidth]{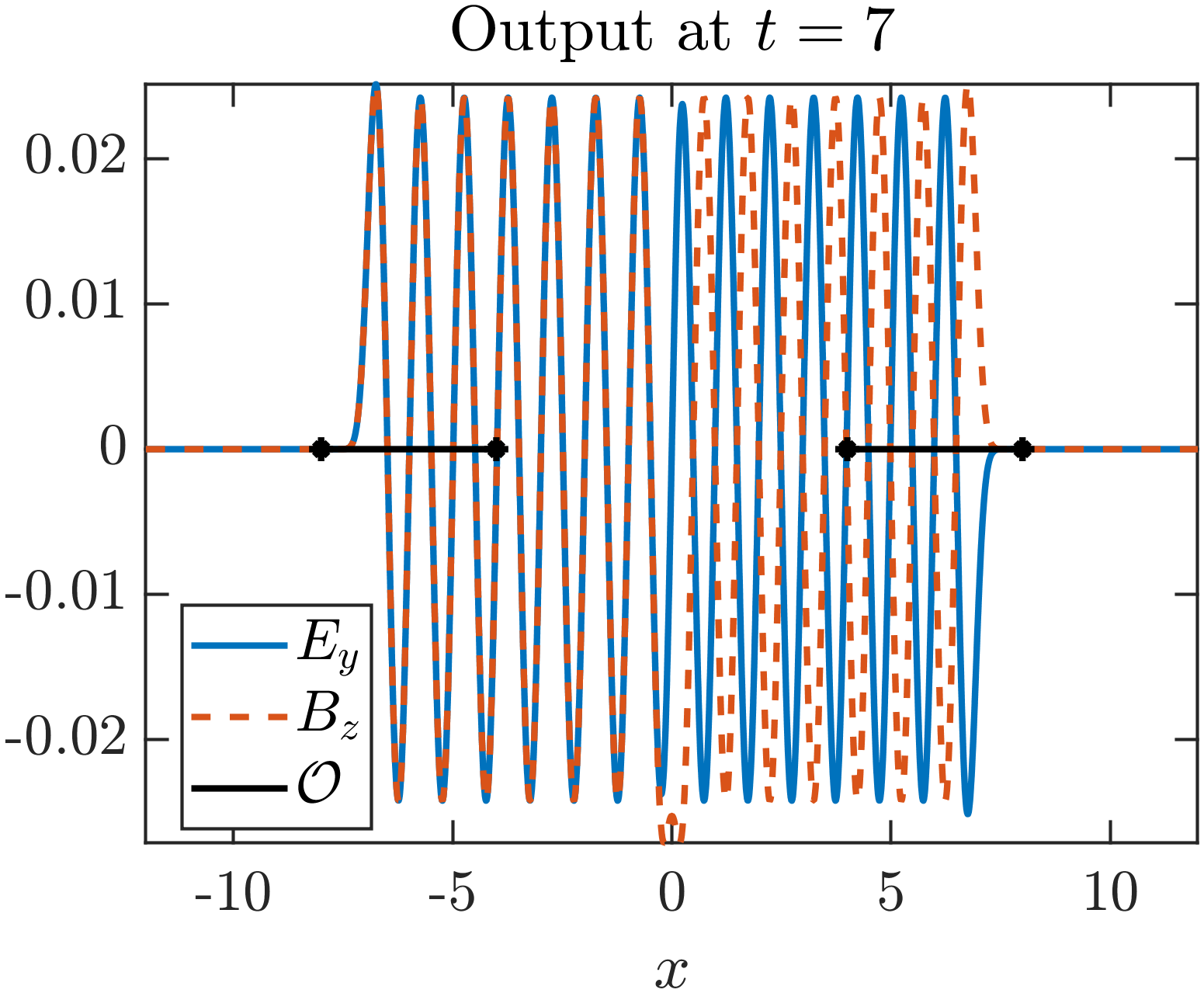}
    \includegraphics[width=0.48\linewidth]{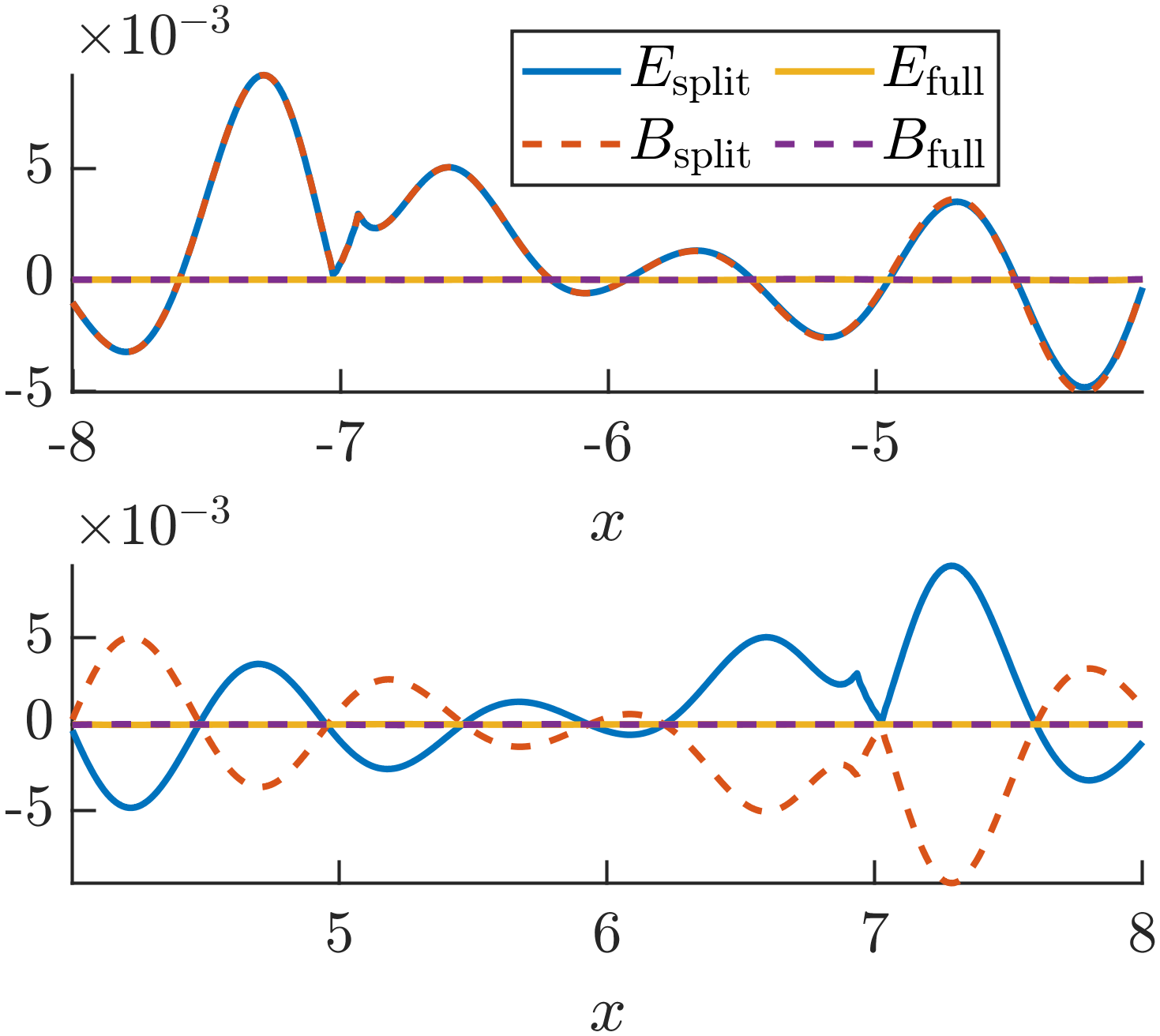}
    \caption{A snapshot of electric and magnetic field components, propagating transversely along the $x$-axis, consistent with a one-dimensional realization of the source cloaking problem with $E_d=0$ and $B_d=0$. The left panel shows the uncontrolled propagation of electromagnetic waves, radiating from a source $I_\text{src}$ oscillating sinusoidally in time and localized at the origin, with solid horizontal lines indicating observation regions $\mathcal{O}=\Omega_{\text{obs}}$. The right panels show the controlled dynamics in each of these observation regions. We see that if the control is assumed to be in a space-time factored form (fields labeled by `split'), results are clearly suboptimal with respect to our formulation (fields labeled by `full'). Here we used a standard time-stepped pseudospectral discretization~\cite{Trefethen2000} and a backtracking line search to execute the numerical optimization \cite{NocedalWright2006}.}
    \label{fig:split_vs_nosplit}
\end{figure}

\paragraph{Organization.}
Section~\ref{sec:background} records the assumptions, notation, and the weak formulation. Section~\ref{sec:forward-num} introduces the N\'ed\'elec/RT spatial discretization and CN time stepping, establishing well-posedness, the exact energy identity, discrete Gauss-law preservation, and convergence of the fully discrete forward scheme. Section~\ref{sec:ocp} studies the optimal control problem and proves well-posedness and continuity of the control-to-state map. Sections~\ref{sec:adjoint} and \ref{sec:foc} derive the adjoint and first-order optimality conditions. Section~\ref{sec:discrete-ocp} presents the fully discrete state/adjoint discretization, the consistent discrete gradient, and convergence of discrete optimal controls. Numerical illustrations of source cloaking are presented in Section~\ref{sec:numerics}.

\section{Background and preliminaries}
\label{sec:background}

Let $\Omega\subset\mathbb{R}^3$ be a bounded Lipschitz domain and let $T>0$ be fixed. 
We study the time–dependent Maxwell's equations with spatially varying (possibly anisotropic) material tensors,
\begin{subequations}\label{eq:Maxwell_strong}
\begin{align}
\varepsilon(x)\,\partial_t E - \curl\big(\mu^{-1}(x) B\big) + \sigma(x)\,E &= f(x,t)
&&\text{in }\Omega\times(0,T), \label{eq:max1}\\
\partial_t B + \curl E &= 0
&&\text{in }\Omega\times(0,T). \label{eq:max2}
\end{align}
\end{subequations}
The system is equipped with initial data and PEC boundary conditions,
\[
E(0,x)=E_0(x),\qquad B(0,x)=B_0(x)\quad\text{in }\Omega,
\qquad E\times \nu=0\quad\text{on }\partial\Omega\times(0,T).
\]
where $\nu$ is the outward-pointing vector normal to $\partial \Omega$.

\begin{assumption}\label{ass:data}
Throughout the paper we suppose: 
\begin{itemize}
  \item The material tensors $\varepsilon(x),\sigma(x)$, and $\mu(x)\in L^\infty(\Omega;\mathbb{R}^{3\times 3})$ are symmetric;
  \item there exist constants $\varepsilon_0,\mu_0>0$ such that, for a.e.\ $x\in\Omega$ and all $\xi\in\mathbb{R}^3$,
  \[
    \xi^\top \varepsilon(x)\xi \ge \varepsilon_0 |\xi|^2,
    \qquad
    \xi^\top \mu(x)\xi \ge \mu_0 |\xi|^2;
  \]
  \item $\sigma$ is symmetric positive semidefinite, $f\in L^2(0,T;L^2(\Omega)^3)$, and $E_0,B_0\in L^2(\Omega)^3$ with $\divv B_0=0$ in $L^2(\Omega)$.
\end{itemize}
\end{assumption}

Throughout, we write $(\cdot,\cdot)$ to denote the $L^2(\Omega)$-inner product on the spatial domain $\Omega$, for both scalar-
or vector-valued functions as dictated by the context.
$\|\cdot\|_{L^2(\Omega)}$ denotes the corresponding norm.
For a Banach space $X$, we denote its topological dual by $X^*$ and write
$\langle\cdot,\cdot\rangle_{X^*,X}$ for the duality pairing between $X^*$ and $X$. The Bochner space $L^p\big(0,T;X\big)$ represents the space of functions which are $L^p(0,T)$ and map to the space $X$. 
Moreover, $\mathcal{D} = C_c^\infty(\Omega)$ is the space of infinitely differentiable functions with compact support in $\Omega$ and $\mathcal D'$ the space of distributions on $\mathcal{D}$.  
We recall the standard Sobolev spaces
\[
\begin{aligned}
H(\divv;\Omega) &:= \{v\in L^2(\Omega)^3:\ \divv v\in L^2(\Omega)\},\\
H_0(\divv;\Omega) &:= \{v\in H(\divv;\Omega):\ \gamma_\nu(v)=v\cdot\nu=0\text{ on }\partial\Omega\},\\
H(\divv^0;\Omega) &:= \{v\in H(\divv;\Omega):\ \divv v=0\ \text{a.e.\ in }\Omega\},\\
H(\curl;\Omega) &:= \{v\in L^2(\Omega)^3:\ \curl v\in L^2(\Omega)^3\},\\
H_0(\curl;\Omega) &:= \{v\in H(\curl;\Omega):\ \gamma_\tau(v)=v\times\nu=0\text{ on }\partial\Omega\}.
\end{aligned}
\]
Here $\gamma_\nu$ and $\gamma_\tau$ are the normal and tangential trace operators; see, e.g., \cite{VGirault_PARaviart_1986a} and \cite[Thm.~5.4]{ABuffa_PJCiarlet_2001a}. For boundary terms involving tangential traces, we use the standard boundary duality pairing $\langle\cdot,\cdot\rangle_{\partial\Omega}$.

\begin{definition}[Weak solution]\label{def:weak_soln}
We call $(E,B)$ a weak solution of~\eqref{eq:Maxwell_strong} if
\[
E \in H^1\!\big(0,T;H_0(\curl;\Omega)^*\big)\cap C([0,T];L^2(\Omega)^3),
\qquad
B \in H^1\!\big(0,T;H(\curl;\Omega)^*\big)\cap C([0,T];L^2(\Omega)^3),
\]
and, for $t$ a.e.\ in $(0,T)$,
\begin{subequations}\label{eq:MaxWeak}
\begin{align}
\langle \varepsilon\,\partial_t E,\,\psi\rangle_{H_0(\curl;\Omega)^*,H_0(\curl;\Omega)} 
+ (\sigma E,\psi) - (\mu^{-1}B,\curl\psi) &= (f,\psi),
&&\forall\,\psi\in H_0(\curl;\Omega),\\
\langle \partial_t B,\,\phi\rangle_{H(\curl;\Omega)^*,H(\curl;\Omega)} + (E,\curl\phi) &= 0,
&&\forall\,\phi\in H(\curl;\Omega),
\end{align}
\end{subequations}
with initial conditions $E(0)=E_0$, $B(0)=B_0$ in $L^2(\Omega)^3$. 
If, in addition, $\divv B_0=0$ in $L^2(\Omega)$, then $B\in L^\infty(0,T;H(\divv^0;\Omega))$.
\end{definition}

The well-posedness of \eqref{eq:Maxwell_strong} in the sense of Definition~\ref{def:weak_soln} under Assumption~\ref{ass:data} has been established in \cite{HAntil_2025a}, see also \cite[Eqns.~(4.31)--(4.32), p.~346; Thm.~4.1]{GDuvaut_JLLions_1976a}, \cite[Thm.~2.4]{AKirsch_ARieder_2016a}, \cite[Lemma~3.2]{IYousept_2020a}, \cite[Sec.~8.2]{RLeis_1997a} and \cite[Sec.~7.8]{MFabrizio_MMorro_2003a}. 
For completeness, we state the result from \cite{HAntil_2025a}. 

\begin{theorem}[Existence, uniqueness, and stability]\label{thm:main}
Let $\Omega\subset\mathbb{R}^3$ be bounded and Lipschitz and $T>0$. 
Under Assumption~\ref{ass:data}, there exists a unique weak solution $(E,B)$ of 
\eqref{eq:Maxwell_strong} in the sense of Definition~\ref{def:weak_soln}. 
Moreover, the solution satisfies
\begin{equation}\label{eq:apriori_bound}
\begin{aligned}
\|\partial_t E\|_{L^2(0,T;H_0(\curl;\Omega)^*)}
&+ \|\partial_t B\|_{L^2(0,T;H(\curl;\Omega)^*)}
+ \|E\|_{C([0,T];L^2(\Omega))} + \|B\|_{C([0,T];L^2(\Omega))}\\
&\le C\Big(\|f\|_{L^2(0,T;L^2(\Omega))}+\|E_0\|_{L^2(\Omega)}+\|B_0\|_{L^2(\Omega)}\Big),
\end{aligned}
\end{equation}
where $C>0$ depends only on $T$ and the bounds/ellipticity of $\varepsilon,\mu,\sigma$.
\end{theorem}

\section{Numerical method for the forward problem}
\label{sec:forward-num}

We develop a structure-preserving scheme for the solution of~\eqref{eq:MaxWeak} and analyze its convergence.
We refer to \cite{HAntil_2025a} for the semi-discrete case.
 
\subsection{Spatial discretization}

Let $\{\mathcal T_h\}_h$ denote a family of shape-regular meshes of $\Omega$, and fix a polynomial degree $k\ge 0$.

\begin{itemize}
  \item \textbf{N\'ed\'elec edge space (conforming to $H(\curl)$).}
  \[
    \mathcal N_h
    := \bigl\{ v_h\in H(\curl;\Omega)\;:\; v_h|_K \in \mathcal N_k(K)\ \text{for all }K\in\mathcal T_h \bigr\},
  \] 
  where $\mathcal N_k(K)$ refers to the local N\'ed\'elec space of the first kind with polynomial degree $k$~\cite[pg. 92]{boffibrezzifortinMixedFEM}. Moreover, we define
  \[\mathcal N_h^0 := \mathcal N_h\cap H_0(\curl;\Omega).\]

  \item \textbf{Raviart--Thomas face space (conforming to $H(\divv)$).}
  \[
    \mathcal{RT}_h
    := \bigl\{ w_h\in H(\divv;\Omega)\;:\; w_h|_K \in \mathcal{RT}_k(K)\ \text{for all }K\in\mathcal T_h \bigr\}.
  \]
    $\mathcal{RT}_k(K)$ refers to the local Raviart--Thomas space with polynomial degree $k$~\cite[pg. 85]{boffibrezzifortinMixedFEM}.

  \item \textbf{Elementwise polynomial scalars.}
  \[
    \mathcal Q_h
    := \bigl\{ q_h\in L^2(\Omega)\;:\; q_h|_K \in \mathbb P_k(K)\ \text{for all }K\in\mathcal T_h \bigr\}.
  \]
   where $\mathbb P_k(K)$ is the space of scalar polynomials on $K$ of degree at most $k$.
\end{itemize}

\paragraph{Discrete de Rham complex and commuting structure.}

The standard continuous sequence
\[
H_0^1(\Omega)\xrightarrow{\ \nabla\ } H_0(\curl;\Omega)\xrightarrow{\ \curl\ } H(\divv;\Omega)\xrightarrow{\ \divv\ } L^2(\Omega)\to 0
\]
is exact~\cite{phillips-shadid2018maxwell}. Its finite element counterpart reads
\[
\mathcal P_h \xrightarrow{\ \nabla\ } \mathcal N_h^0 \xrightarrow{\ \curl\ } \mathcal{RT}_h \xrightarrow{\ \divv\ } \mathcal Q_h \to 0,
\]
where $\mathcal P_h$ denotes the conforming Lagrange space of degree $k$. In particular,
\[
\curl\,\mathcal N_h^0 \subset \mathcal{RT}_h,
\qquad
\divv\bigl(\curl\,\mathcal N_h^0\bigr)=0.
\]

\paragraph{Elementwise $L^2$ projection.}
Let $\Pi_0:L^2(\Omega)^3\to\mathcal Q_h$ be the $L^2$–orthogonal projector (defined elementwise). It is $L^2$–stable:
\[
\|\Pi_0 v\|_{L^2(\Omega)} \le \|v\|_{L^2(\Omega)}
\qquad \forall\, v\in L^2(\Omega)^3.
\]

\subsection{Time discretization}

Let $N_t$ denote the number of time steps, $\Delta t:=T/N_t$, and $t^n:=n\,\Delta t$ for $0\le n\le N_t$.
Denote the time grid by $\mathcal T_t:=\{t^n\}_{n=0}^{N_t}$. For $g\in C([0,T];X)$ set $g^n:=g(t^n)$ and
$g^{\Delta t}:=\{g^n\}_{n=0}^{N_t}$. On sequences $g^{\Delta t}\subset X$ define
\[
  \|g^{\Delta t}\|_{\ell^\infty(X)}:=\max_{0\le n\le N_t}\|g^n\|_X,
  \qquad
  \|g^{\Delta t}\|_{\ell^2(X)}^2:=\sum_{{n=0}}^{N_t}\Delta t\,\|g^n\|_X^2.
\]
For $n=0,\dots,N_t-1$ we use the backward difference and the midpoint average
\[
  \delta_t g^{n+1}:=\frac{g^{n+1}-g^n}{\Delta t},
  \qquad
  g^{n+\frac12}:=\frac{g^{n+1}+g^n}{2}.
\]

\paragraph{Piecewise constant reconstructions.}
We identify a nodal sequence with a right–continuous piecewise constant function,
\[
  g^{\mathrm{pc}}(t):=g^n \quad \text{for } t\in(t^{n-1},t^n],\ \ n=1,\dots,N_t,
\]
so that $\|g^{\mathrm{pc}}\|_{L^\infty(0,T;X)}=\|g^{\Delta t}\|_{\ell^\infty(X)}$ and
$\|g^{\mathrm{pc}}\|_{L^2(0,T;X)}^2=\sum_{{n=0}}^{N_t}\Delta t\,\|g^n\|_X^2$.
Likewise, for midpoints define the (right–continuous) midpoint lift
\[
  (g^{1/2})(t):=g^{n+\frac12}\quad\text{for }t\in(t^n,t^{n+1}],\ \ n=0,\dots,N_t-1,
\]
with
\[
  \|g^{1/2}\|_{L^\infty(0,T;X)}=\max_{0\le n\le N_t-1}\|g^{n+\frac12}\|_X,
  \qquad
  \|g^{1/2}\|_{L^2(0,T;X)}^2=\sum_{n=0}^{N_t-1}\Delta t\,\|g^{n+\frac12}\|_X^2.
\]
We also lift the discrete time derivative as a piecewise constant function
\[
  (\partial_t^{\Delta t} g)(t):=\delta_t g^{n+1}\quad\text{for }t\in(t^n,t^{n+1}],\ \ n=0,\dots,N_t-1,
\]
so that $\|\,\partial_t^{\Delta t} g\,\|_{L^2(0,T;X)}^2=\sum_{n=0}^{N_t-1}\Delta t\,\|\delta_t g^{n+1}\|_X^2$.

\paragraph{Piecewise affine (linear) reconstruction.}
We also use the continuous, piecewise affine reconstruction $g^{\mathrm{pl}}:(0,T]\to X$ defined by
\[
  g^{\mathrm{pl}}(t):=g^n+(t-t^n)\,\delta_t g^{n+1}
  \qquad \text{for } t\in(t^n,t^{n+1}],\ \ n=0,\dots,N_t-1.
\]
Then $g^{\mathrm{pl}}\in C([0,T];X)$, $g^{\mathrm{pl}}(t^n)=g^n$, and
\[
  \partial_t g^{\mathrm{pl}}(t)=\delta_t g^{n+1}\quad\text{for }t\in(t^n,t^{n+1}],
  \qquad
  g^{\mathrm{pl}}(t^{n+\frac12})=g^{n+\frac12}.
\]
In particular,
\[
  \|\partial_t g^{\mathrm{pl}}\|_{L^2(0,T;X)}^2=\sum_{n=0}^{N_t-1}\Delta t\,\|\delta_t g^{n+1}\|_X^2.
\]

\subsection{Fully discrete Crank--Nicolson scheme}

For $f\in L^2(0,T;L^2(\Omega)^3)$ we define the Crank--Nicolson midpoint value on
$I_n=(t^n,t^{n+1})$ by the cell average
\[
  f^{\,n+\frac12}
  := \frac{1}{\Delta t}\int_{t^n}^{t^{n+1}} f(s)\,ds,
  \qquad n=0,\dots,N_t-1.
\]

\paragraph{Fully discrete CN step.}
Given $(E_h^n,B_h^n)\in \mathcal N_h^0\times \mathcal{RT}_h$, find $(E_h^{n+1},B_h^{n+1})\in \mathcal N_h^0\times \mathcal{RT}_h$ such that
for all $(\psi_h,\phi_h)\in \mathcal N_h^0\times \mathcal{RT}_h$,
\begin{subequations}\label{eq:CN}
\begin{align}
(\varepsilon\,\delta_t E_h^{n+1},\psi_h)
+(\sigma\, E_h^{n+\frac12},\psi_h)
-(\mu^{-1} B_h^{n+\frac12},\curl\,\psi_h)
&=(f^{\,n+\frac12},\psi_h),\label{eq:CN_ampere}\\
(\delta_t B_h^{n+1},\phi_h)
+(\curl\, E_h^{n+\frac12},\phi_h)&=0,\label{eq:CN_faraday}
\end{align}
\end{subequations}
for $n=0,\dots,N_t-1$.

\paragraph{Initial data.}
We take
\begin{equation}\label{eq:initialdiscdata}
E_h(0)=Q_h^N E_0 \in \mathcal N_h^0,\qquad
B_h(0)\in \mathcal{RT}_h\ \text{with}\ (\divv B_h(0),q_h)=0\ \ \forall q_h\in \mathcal Q_h.
\end{equation}
Here, $Q_h^N: L^2(\Omega)^3\to \mathcal N_h^0$ is the $L^2$–orthogonal projector, so that
$\|E_h(0)-E_0\|_{L^2(\Omega)}\!\to 0$ and 
choosing $B_h(0)$ as the $L^2$–orthogonal projection of
$B_0$ onto the discrete divergence–free subspace
\(Z_h := \{v_h\in\mathcal{RT}_h:\ \divv v_h=0\ \text{in }\mathcal Q_h\}\)
gives $\|B_h(0)-B_0\|_{L^2(\Omega)}\!\to 0$ as $h\to0$. See \cite[Lemma~7 \& 8]{HAntil_2025a} for the
convergence of $E_h(0)$ and $B_h(0)$ in $L^2(\Omega)^3$.

\begin{lemma}[CN well-posedness; energy identity; stability for nonnegative $\sigma$]\label{lem:CN-wp}
Let Assumption~\ref{ass:data} hold.
Let the discrete energy be
\begin{equation}\label{eq:disc_energy}
\mathcal{E}^n
:=\tfrac12\Big(\|\varepsilon^\frac12\,E_h^{n}\|_{L^2(\Omega)}^2
+\|\mu^{-\frac12}\,B_h^{n}\|_{L^2(\Omega)}^2\Big).
\end{equation}
Given $(E_h^n,B_h^n)\in \mathcal N_h^0\times \mathcal{RT}_h$, problem \eqref{eq:CN} admits a unique solution
$(E_h^{n+1},B_h^{n+1})\in \mathcal N_h^0\times \mathcal{RT}_h$, and the solution depends continuously on the data
$(E_h^n,B_h^n,f^{n+\frac12})$.

Moreover, for the energy \eqref{eq:disc_energy} the \emph{exact balance}
\begin{equation}\label{eq:CN_energy_identity}
\frac{\mathcal{E}^{n+1}-\mathcal{E}^{n}}{\Delta t}
\;+\;\|\sqrt{\sigma}\,E_h^{n+\frac12}\|_{L^2(\Omega)}^2
\;=\;
\big(f^{\,n+\frac12},\,E_h^{n+\frac12}\big) 
\end{equation}
holds for every $n$. In particular, if $f\equiv0$ then the scheme is energy-dissipative for $\sigma\ge0$ (and strictly conservative when $\sigma\equiv0$).

Finally, there exists a constant $C_T>0$ depending only on $T$, the ellipticity bounds of $\varepsilon,\mu$, and $\|\varepsilon\|_{L^\infty}$,$\|\mu^{-1}\|_{L^\infty}$ (but not on $h,\Delta t$), such that for all $m=1,\dots,N_t$,
\begin{equation}\label{eq:CN_stability_eps}
\mathcal{E}^{m}
\;+\;\Delta t\,\sum_{n=0}^{m-1}\|\sqrt{\sigma}\,E_h^{n+\frac12}\|_{L^2(\Omega)}^2
\;\le\;
C_T\Bigg(\,
\mathcal{E}^{0}
\;+\;\sum_{n=0}^{m-1}\Delta t\,\big\|\varepsilon^{-1/2}f^{\,n+\frac12}\big\|_{L^2(\Omega)}^{2}
\Bigg).
\end{equation}
\end{lemma}
\begin{proof}
\emph{Existence and uniqueness.}
Equations~\eqref{eq:CN} define a square linear system. To prove that a solution exists and that it is unique, it is sufficient to show that the system operator is injective.  We use a proof by contradiction.
Let two solutions $(E_1^{n+1}, B_1^{n+1})$ be given for the same data $(E^n_h,B_h^n,f^{n+\frac12})$. Set the differences
\[
\widehat{E}^{n+1} := E_1^{n+1} - E_2^{n+1} , 
\quad 
\widehat{B}^{n+1} := B_1^{n+1} - B_2^{n+1} , 
\]
and note that $\widehat E^{n}=\widehat B^{n}=0$ (the $n$-level values are given and equal). The difference satisfies the homogeneous CN system
for $(\widehat E^{n+1},\widehat B^{n+1})$:
\begin{subequations}
\begin{align}
\Big(\varepsilon\,\frac{\widehat E^{n+1}-\widehat E^{n}}{\Delta t},\,\psi_h\Big)
+(\sigma \widehat E^{n+\frac12},\psi_h)
-(\mu^{-1}\widehat B^{n+\frac12},\mathrm{curl}\,\psi_h)
&=0 \quad \forall\,\psi_h\in \mathcal N_h^0, \label{eq:CN-hom-A}\\
\Big(\frac{\widehat B^{n+1}-\widehat B^{n}}{\Delta t},\,\phi_h\Big)
+(\mathrm{curl}\,\widehat E^{n+\frac12},\phi_h)&=0 \quad \forall\,\phi_h\in \mathcal{RT}_h. 
\label{eq:CN-hom-F}
\end{align}
\end{subequations}
In view of the definition of $L^2$-orthogonal projection $P_h$ (see Proposition~\ref{prop:L2proj-RT} for details), choose the specific test functions $\psi_h=\widehat E^{n+\frac12}:=\tfrac12(\widehat E^{n+1}+\widehat E^{n})=\tfrac12\widehat E^{n+1}$ (since $\widehat E^{n}=0$)
and $\phi_h=P_h\left(\mu^{-1}\widehat B^{n+\frac12}\right):=P_h\left(\mu^{-1}\tfrac12(\widehat B^{n+1}+\widehat B^{n})\right)
=\tfrac12\mu^{-1}\widehat B^{n+1}$ (since $\widehat B^{n}=0$).
Then \eqref{eq:CN-hom-A} gives
\begin{equation}\label{eq:uniq-A}
\Big(\varepsilon\,\frac{\widehat E^{n+1}-0}{\Delta t},\,\tfrac12\widehat E^{n+1}\Big)
+(\sigma \tfrac12\widehat E^{n+1},\tfrac12\widehat E^{n+1})
-(\mu^{-1}\tfrac12\widehat B^{n+1},\mathrm{curl}\,\tfrac12\widehat E^{n+1})=0,
\end{equation}
i.e.,
\begin{equation}\label{eq:uniq-A2}
\frac{1}{2\Delta t}\|\sqrt{\varepsilon}\,\widehat E^{n+1}\|_{L^2(\Omega)}^2
+\frac14\|\sqrt{\sigma}\,\widehat E^{n+1}\|_{L^2(\Omega)}^2
-\frac14(\mu^{-1}\widehat B^{n+1},\mathrm{curl}\,\widehat E^{n+1})=0.
\end{equation}
Similarly, \eqref{eq:CN-hom-F} with $\phi_h=P_h\left(\mu^{-1}\widehat B^{n+\frac12}\right)$ gives
\begin{equation}\label{eq:uniq-F0}
\Big(\frac{\widehat B^{n+1}-0}{\Delta t},\,\mu^{-1}\tfrac12\widehat B^{n+1}\Big)
+(\mathrm{curl}\,\widehat E^{n+\frac12},P_h(\mu^{-1}\widehat B^{n+\frac12}))=0,
\end{equation}
Since $\mathrm{curl}\,\widehat E^{n+\frac12} \in \RT_h$, from the orthogonality property of $P_h$ (see Proposition~\ref{prop:L2proj-RT}),  \eqref{eq:uniq-F0} is equivalent to 
\begin{equation}\label{eq:uniq-F}
\Big(\frac{\widehat B^{n+1}-0}{\Delta t},\,\mu^{-1}\tfrac12\widehat B^{n+1}\Big)
+(\mathrm{curl}\,\widehat E^{n+\frac12},\mu^{-1}\widehat B^{n+\frac12})=0,
\end{equation}
i.e.,
\begin{equation}\label{eq:uniq-F2}
\frac{1}{2\Delta t}\|\sqrt{\mu^{-1}}\,\widehat B^{n+1}\|_{L^2(\Omega)}^2
+(\mathrm{curl}\,\tfrac12\widehat E^{n+1},\mu^{-1}\tfrac12\widehat B^{n+1})=0.
\end{equation}
Hence the second term in \eqref{eq:uniq-F2} cancels the last term in
\eqref{eq:uniq-A2} upon summation. Adding \eqref{eq:uniq-A2} and \eqref{eq:uniq-F2} we obtain
\[
\frac{1}{2\Delta t}\|\sqrt{\varepsilon}\,\widehat E^{n+1}\|^2_{L^2(\Omega)}
+\frac14\|\sqrt{\sigma}\,\widehat E^{n+1}\|^2_{L^2(\Omega)}
+\frac{1}{2\Delta t}\|\sqrt{\mu^{-1}}\,\widehat B^{n+1}\|^2_{L^2(\Omega)}
=0.
\]
Since $\varepsilon,\mu^{-1}$ are uniformly elliptic, all three terms are nonnegative; hence they must vanish. In particular $\widehat E^{n+1}=0$ and $\widehat B^{n+1}=0$, yielding a unique pair $(E_h^{n+1},B_h^{n+1})\in \mathcal N_h^0\times \mathcal{RT}_h$.

\emph{Energy identity.}
We test \eqref{eq:CN_ampere} with $\psi_h=E_h^{n+\frac12} \in \mathcal N_h^0 $ and \eqref{eq:CN_faraday} with
$\phi_h=P_h\left(\mu^{-1}B_h^{n+\frac12}\right) \in \mathcal{RT}_h$. Using the algebraic identities, valid for any symmetric positive definite matrix~$M$ and vectors~$u,v$,
\[
\Big(M\frac{u-v}{\Delta t},\,\frac{u+v}{2}\Big)
=\frac{1}{2\Delta t}\Big(\|\sqrt{M}u\|_{L^2(\Omega)}^2-\|\sqrt{M}v\|_{L^2(\Omega)}^2\Big),
\qquad
(Ma,a)=\|\sqrt{M}a\|^2,
\]
we obtain
\begin{subequations}
\begin{align}
\frac{1}{2\Delta t}\!\Big(\|\sqrt{\varepsilon}E_h^{n+1}\|_{L^2(\Omega)}^2-\|\sqrt{\varepsilon}E_h^{n}\|_{L^2(\Omega)}^2\Big)
+ \|\sqrt{\sigma}E_h^{n+\frac12}\|_{L^2(\Omega)}^2
-(\mu^{-1}B_h^{n+\frac12},\mathrm{curl}\,E_h^{n+\frac12})
&=(f^{n+\frac12},E_h^{n+\frac12}),
\label{eq:en-A}\\
\frac{1}{2\Delta t}\!\Big(\|\sqrt{\mu^{-1}}B_h^{n+1}\|_{L^2(\Omega)}^2-\|\sqrt{\mu^{-1}}B_h^{n}\|_{L^2(\Omega)}^2\Big)
+ (\mathrm{curl}\, E_h^{n+\frac12},\mu^{-1}B_h^{n+\frac12})
&=0,
\label{eq:en-F}
\end{align}
\end{subequations}
where \eqref{eq:en-F} is due to the definition of $P_h$, giving 
$(\mathrm{curl}\, E_h^{n+\frac12},P_h(\mu^{-1}B_h^{n+\frac12})) = (\mathrm{curl}\, E_h^{n+\frac12},\mu^{-1}B_h^{n+\frac12})$.
Summing \eqref{eq:en-A} and \eqref{eq:en-F} yields the \emph{per-step CN energy balance} \eqref{eq:CN_energy_identity}
\begin{equation}\label{eq:en-balance}
\frac{\mathcal{E}^{\,n+1}-\mathcal{E}^{\,n}}{\Delta t}
+\|\sqrt{\sigma}\,E_h^{n+\frac12}\|_{L^2(\Omega)}^2
= (f^{n+\frac12},E_h^{n+\frac12}) . 
\end{equation}
In particular, if $f\equiv0$, then $\mathcal E^{n+1}\le \mathcal E^n$ (and equality when $\sigma\equiv0$).

\emph{Stability for $\sigma\ge0$ (discrete Gr\"onwall's inequality on energy dissipation).}
From \eqref{eq:CN_energy_identity} and the Cauchy--Schwarz/Young inequalities with a parameter $\gamma>0$,
\[
\frac{\mathcal{E}^{n+1}-\mathcal{E}^{n}}{\Delta t}
+\|\sqrt{\sigma}E_h^{n+\frac12}\|_{L^2(\Omega)}^2
\;\le\;
\frac{1}{2\gamma}\|\varepsilon^{-1/2}f^{n+\frac12}\|_{L^2(\Omega)}^2
+\frac{\gamma}{2}\,\|\sqrt{\varepsilon}E_h^{n+\frac12}\|_{L^2(\Omega)}^2.
\]
Using
\(
\|\sqrt\varepsilon E_h^{n+\frac12}\|_{L^2(\Omega)}^2
=\tfrac12\big(\|\sqrt\varepsilon E_h^{n+1}\|_{L^2(\Omega)}^2+\|\sqrt\varepsilon E_h^{n}\|_{L^2(\Omega)}^2\big)
-\tfrac14\|\sqrt\varepsilon(E_h^{n+1}-E_h^{n})\|_{L^2(\Omega)}^2
\le \mathcal E^{n+1}+\mathcal E^{n},
\)
we obtain 
\[
\frac{\mathcal{E}^{n+1}-\mathcal{E}^{n}}{\Delta t}
+\|\sqrt{\sigma}E_h^{n+\frac12}\|_{L^2(\Omega)}^2
\;\le\;
\frac{\gamma}{2}\big(\mathcal E^{n+1}+\mathcal E^{n}\big)
+\frac{1}{2\gamma}\|\varepsilon^{-1/2}f^{n+\frac12}\|_{L^2(\Omega)}^2.
\]
After multiplying by $\Delta t$ and rearranging,
\[
\bigl(1-\tfrac{\gamma\Delta t}{2}\bigr)\,\mathcal E^{n+1}
\;+\;\Delta t\,\|\sqrt{\sigma}E_h^{n+\frac12}\|_{L^2(\Omega)}^2
\;\le\;
\bigl(1+\tfrac{\gamma\Delta t}{2}\bigr)\,\mathcal E^{n}
+\frac{\Delta t}{2\gamma}\,\|\varepsilon^{-1/2}f^{n+\frac12}\|_{L^2(\Omega)}^2.
\]
Let $\alpha:=\gamma\Delta t/2$. Starting from
\[
(1-\alpha)\,\mathcal E^{n+1}
+\Delta t\,\|\sqrt{\sigma}E_h^{n+\frac12}\|_{L^2(\Omega)}^2
\le
(1+\alpha)\,\mathcal E^{n}
+\frac{\Delta t}{2\gamma}\,\|\varepsilon^{-1/2}f^{n+\frac12}\|_{L^2(\Omega)}^2,
\]
we add $\sum_{k=0}^{n-1}\Delta t\,\|\sqrt{\sigma}E_h^{k+\frac12}\|_{L^2(\Omega)}^2$ to both sides and define
\[
S_n:=\mathcal E^{\,n}+\sum_{k=0}^{n-1}\Delta t\,\|\sqrt{\sigma}E_h^{k+\frac12}\|_{L^2(\Omega)}^2
\qquad (S_0=\mathcal E^0).
\]
Since $\sum_{k=0}^{n}\Delta t\,\|\sqrt{\sigma}E_h^{k+\frac12}\|_{L^2(\Omega)}^2\ge0$, we have
\[
(1-\alpha)\,\mathcal E^{n+1}+\sum_{k=0}^{n}\Delta t\,\|\sqrt{\sigma}E_h^{k+\frac12}\|_{L^2(\Omega)}^2
=(1-\alpha)\,S_{n+1}+\alpha\sum_{k=0}^{n}\Delta t\,\|\sqrt{\sigma}E_h^{k+\frac12}\|_{L^2(\Omega)}^2
\ge (1-\alpha)\,S_{n+1},
\]
and clearly
\[
(1+\alpha)\,\mathcal E^{n}+\sum_{k=0}^{n-1}\Delta t\,\|\sqrt{\sigma}E_h^{k+\frac12}\|_{L^2(\Omega)}^2
\le (1+\alpha)\,S_n.
\]
Therefore,
\[
(1-\alpha)\,S_{n+1}
\ \le\
(1+\alpha)\,S_n
+\frac{\Delta t}{2\gamma}\,\|\varepsilon^{-1/2}f^{n+\frac12}\|_{L^2(\Omega)}^2.
\]
We divide by $(1-\alpha)$ and set $\rho:=\frac{1+\alpha}{1-\alpha}$, $c:=\frac{1}{2\gamma(1-\alpha)}$, to get
\[
S_{n+1}\ \le\ \rho\,S_n\ +\ c\,\Delta t\,\|\varepsilon^{-1/2}f^{n+\frac12}\|_{L^2(\Omega)}^2.
\]
We multiply by $\rho^{-(n+1)}$ and sum $n=0,\dots,m-1$; since the telescoping sum
$\sum_{n=0}^{m-1}\big(\rho^{-(n+1)}S_{n+1}-\rho^{-n}S_n\big)$ reduces to $\rho^{-m}S_m-S_0$, we obtain
\[
\rho^{-m}S_m - S_0 \ \le\ c\sum_{n=0}^{m-1}\rho^{-(n+1)}\Delta t\,\|\varepsilon^{-1/2}f^{n+\frac12}\|_{L^2(\Omega)}^2
\ \le\ c\sum_{n=0}^{m-1}\Delta t\,\|\varepsilon^{-1/2}f^{n+\frac12}\|_{L^2(\Omega)}^2.
\]
Hence
\[
S_m\ \le\ \rho^{\,m}\Big(S_0 + c\sum_{n=0}^{m-1}\Delta t\,\|\varepsilon^{-1/2}f^{n+\frac12}\|_{L^2(\Omega)}^2\Big).
\]
By choosing $\gamma:=1/(2T)$ so that $\alpha=\Delta t/(4T)\in(0,\tfrac12),$ we may use the inequality
$\rho=\frac{1+\alpha}{1-\alpha}\le e^{4\alpha}$, which can be verified by a Taylor expansion of $\log\frac{1+\alpha}{1-\alpha}$, whence
$\rho^{\,m}\le e^{4\alpha m}=e^{2\gamma t_m}\le e^{2\gamma T}$.
Also note that $1-\alpha\ge\tfrac12$, so $c\le 1/\gamma$.
Therefore
\[
\mathcal E^{\,m}+\sum_{n=0}^{m-1}\Delta t\,\|\sqrt{\sigma}E_h^{n+\frac12}\|_{L^2(\Omega)}^2
\ \le\
e^{2\gamma T}\Big(\mathcal E^0+\tfrac{1}{\gamma}\sum_{n=0}^{m-1}\Delta t\,\|\varepsilon^{-1/2}f^{n+\frac12}\|_{L^2(\Omega)}^2\Big).
\]
With $\gamma=1/(2T)$ this gives
\[
\mathcal E^{\,m}+\sum_{n=0}^{m-1}\Delta t\,\|\sqrt{\sigma}E_h^{n+\frac12}\|_{L^2(\Omega)}^2
\ \le\
e\,\Big(\mathcal E^0+2T\sum_{n=0}^{m-1}\Delta t\,\|\varepsilon^{-1/2}f^{n+\frac12}\|_{L^2(\Omega)}^2\Big),
\]
which, after absorbing $e$ into $C_T$, is the bound given by~\eqref{eq:CN_stability_eps}.
\end{proof}
Next we study the discrete divergence of $B_h$. We write
\[
\mathrm{div}_h := \mathrm{div}\big|_{\mathcal{RT}_h}: \mathcal{RT}_h \to \mathcal{Q}_h,
\]
and stress that $\mathrm{div}_h$ is \emph{only} applied to functions in $\mathcal{RT}_h$; it is not defined on arbitrary $L^2(\Omega)^3$ fields.

\begin{lemma}[Discrete Gauss law for $B$]
\label{lem:disc_gauss}
Let $\{(E_h^{n},B_h^{n})\}$ solve the fully discrete system
\eqref{eq:CN_ampere} and \eqref{eq:CN_faraday}. 
Then the discrete divergence evolves in $\mathcal{Q}_h$ according to
\begin{equation}\label{eq:disc_div_update}
\delta_t\big(\mathrm{div}_h B_h^{n+1}\big) + \mathrm{div}_h\big(\mathrm{curl}\,E_h^{n+\frac12}\big)=0 \quad\text{in } \mathcal{Q}_h, 
\end{equation}
which implies that $\delta_t(\mathrm{div}_h B_h^{n+1})=0$ in $\mathcal{Q}_h$. In particular, if the initial data satisfy $\mathrm{div}_h B_h^{0}=0$ in $\mathcal{Q}_h$, then $\mathrm{div}_h B_h^{n}=0$ in $\mathcal{Q}_h$ for all $n$.
\end{lemma}

\begin{proof}
Let $r_h := \delta_t B_h^{n+1} + \mathrm{curl}\,E_h^{n+\frac12}$.
From \eqref{eq:CN_faraday}, we have that 
\[
(r_h , \phi_h) = 0, \quad \forall \phi_h \in \mathcal{RT}_h , 
\]
i.e., $r_h \in \mathcal{RT}_h^\perp$ (with respect to the $L^2$ inner product).
Since $B_h^{n},B_h^{n+1}\in \mathcal{RT}_h$, by definition $\delta_t B_h^{n+1}\in \mathcal{RT}_h$.
From the sequence $\mathcal N_h^0 \xrightarrow{\ \curl\ } \mathcal{RT}_h \xrightarrow{\ \divv\ } \mathcal Q_h$, we have $\curl \mathcal N_h^0 \subset \mathcal{RT}_h$.
Therefore $\curl E_h^{n+\frac12}\in \mathcal{RT}_h$ implying $r_h \in \mathcal{RT}_h$.
Thus,
\[
r_h \in \mathcal{RT}_h \cap \mathcal{RT}_h^\perp ,
\]
whence $r_h = 0$ in $\mathcal{RT}_h$.
Applying the divergence operator $\mathrm{div}_h:\mathcal{RT}_h\to Q_h$ to $r_h$,
we get \eqref{eq:disc_div_update} in $Q_h$. Whence $\delta_t(\mathrm{div}_h B_h^{n+1})=0$ in $Q_h$.
Clearly, if $\mathrm{div}_h B_h^0=0$, then $\mathrm{div}_h B_h^n=0$ in $\mathcal{Q}_h$ for all $n$.
\end{proof}

\begin{lemma}[Discrete electric Gauss law for $E$]
\label{lem:fully_disc_gauss_E}
Let $(E_h^{n},B_h^{n})$ solve the fully-discrete system \eqref{eq:CN_ampere} and \eqref{eq:CN_faraday}. Define
the discrete divergence functional
\[
\langle \divv_h(\varepsilon E_h^{n+1}), q_h \rangle_{\mathcal P_h^*,\mathcal P_h}
:=-(\varepsilon E_h^{n+1},\nabla q_h)
\qquad \forall q_h\in \mathcal P_h.
\]
Similarly, define
\[
\langle \divv_h(\sigma E_h^{n+\frac12}-f^{n+\frac12}), q_h \rangle_{\mathcal P_h^*,\mathcal P_h}
:=-(\sigma E_h^{n+\frac12}-f^{n+\frac12},\nabla q_h)
\qquad \forall q_h\in \mathcal P_h.
\]
Then, for all $n$,
\[
\delta_t \divv_h(\varepsilon E_h^{n+1})
+
\divv_h(\sigma E_h^{n+\frac12} - f^{\,n+\frac12})
=0
\qquad\text{in }\mathcal P_h^*.
\]
If moreover,
\[
\divv_h(\sigma E_h^{n+\frac12} - f^{\,n+\frac12})=0
\qquad\text{in }\mathcal P_h^*, \quad \text{for all } n
\]
and
\[
\divv_h(\varepsilon E_h^0)=0
\qquad\text{in }\mathcal P_h^*,
\]
then
\[
\divv_h(\varepsilon E_h^{n})=0
\qquad\text{in }\mathcal P_h^*
\quad\text{for all } n.
\]
\end{lemma}
\begin{proof}
The proofs follow along similar lines as in \cite[Lemma 3]{HAntil_2025a} -
since $\nabla \mathcal P_h\subset \mathcal N_h^0$ by the discrete de Rham complex, choose
\( \psi_h=\nabla q_h \) for $q_h\in \mathcal P_h$ 
as a test function in \eqref{eq:CN_ampere}, using $\curl(\nabla q_h)=0$
we obtain
\[
(\varepsilon\,\delta_t E_h^{n+1},\nabla q_h)
+(\sigma\, E_h^{n+\frac12} - f^{n+\frac12},\nabla q_h)
=0
\qquad \forall q_h\in \mathcal P_h.
\]
Using the definition of the discrete divergence functional, this is the same as
\[
\langle \delta_t \divv_h(\varepsilon E_h^{n+1}), q_h \rangle_{\mathcal P_h^*,\mathcal P_h}
+
\langle \divv_h(\sigma E_h^{n+\frac12}-f^{n+\frac12}), q_h \rangle_{\mathcal P_h^*,\mathcal P_h}
=0
\qquad \forall q_h \in \mathcal P_h.
\]
which proves
\[
\delta_t \divv_h(\varepsilon E_h^{n+1})
+
\divv_h(\sigma E_h^{n+\frac12} - f^{\,n+\frac12})
=0
\qquad\text{in }\mathcal P_h^*.
\]
for all $n$. If in addition
\[
\divv_h(\sigma E_h^{n+\frac12} - f^{\,n+\frac12})=0 \qquad\text{in }\mathcal P_h^*
\qquad\text{for all } n,
\]
then
\[
\delta_t \divv_h(\varepsilon E_h^{n+1}) = 0 \qquad\text{in }\mathcal P_h^*
\]
Clearly then, if $\divv_h(\varepsilon E_h^0)=0$, then $\divv_h(\varepsilon E_h^n)=0$ in $\mathcal{P}_h^*$ for all $n$.
\end{proof}

\subsection{Convergence of the discrete solution to the continuous solution}

\paragraph{Lifted (in-time) form.}
Let $E_h^{\mathrm{pl}}$ and $B_h^{\mathrm{pl}}$ be the piecewise–linear time lifts and
$E_h^{1/2},B_h^{1/2},f^{1/2}$ the piecewise–constant midpoint lifts.
Since $\partial_t^{\Delta t}E_h=\partial_t E_h^{\mathrm{pl}}$ and
$\partial_t^{\Delta t}B_h=\partial_t B_h^{\mathrm{pl}}$ a.e.\ in $(0,T)$, the CN system \eqref{eq:CN} is equivalently:
\begin{equation}\label{eq:CN-lift}
\begin{aligned}
(\varepsilon\,\partial_t E_h^{\mathrm{pl}}(t),\,\psi_h)
+(\sigma\,E_h^{1/2}(t),\,\psi_h)
-(\mu^{-1} B_h^{1/2}(t),\,\curl\psi_h)
&=(f^{1/2}(t),\,\psi_h), &&\forall\,\psi_h\in \mathcal N_h^0,\\[2mm]
(\partial_t B_h^{\mathrm{pl}}(t),\,\phi_h)
+(\curl E_h^{1/2}(t),\,\phi_h)&=0, &&\forall\,\phi_h\in \mathcal{RT}_h,
\end{aligned}
\end{equation}
holding for a.e.\ $t\in(0,T)$ (equivalently, for each $n$ and a.e.\ $t\in(t^n,t^{n+1})$).

\begin{lemma}\label{lem:fbound}
Under the assumptions of Lemma~\ref{lem:CN-wp}, the following result holds
\[
\sum_{n=0}^{N_t-1}\Delta t\,\big\|\varepsilon^{-1/2} f^{\,n+\frac12}\big\|_{L^2(\Omega)}^2 
\le C \|f\|^2_{L^2(0,T;L^2(\Omega))} 
\]
where the constant $C > 0$ only depends on $\|\varepsilon^{-1}\|_{L^\infty}$.
\end{lemma}
\begin{proof}
Let $g = \varepsilon^{-1/2} f(t)$, and note that \(
f^{n+\frac12}:=\frac{1}{\Delta t}\int_{t^n}^{t^{n+1}}f(s)\,ds
\).
Since the norm is convex, applying Jensen’s inequality for Bochner integrals on $I_n$ \cite[pp. 120]{MR2759829} yields
\[
\big\|\varepsilon^{-1/2} f^{\,n+\frac12}\big\|_{L^2(\Omega)}^2
=\Big\|\frac{1}{\Delta t}\int_{I_n} g(s)\,ds\Big\|_{L^2(\Omega)}^2
\le \frac{1}{\Delta t}\int_{I_n}\|g(s)\|_{L^2(\Omega)}^2\,ds
=\frac{1}{\Delta t}\int_{I_n}\|\varepsilon^{-1/2}f(s)\|_{L^2(\Omega)}^2\,ds.
\]
Multiplying by $\Delta t$ and summing over $n=0,\dots,N_t-1$ gives
\[
\sum_{n=0}^{N_t-1}\Delta t\,\big\|\varepsilon^{-1/2} f^{\,n+\frac12}\big\|_{L^2(\Omega)}^2
\;\le\;\sum_{n=0}^{N_t-1}\int_{I_n}\|\varepsilon^{-1/2}f(t)\|_{L^2(\Omega)}^2\,dt
=\int_0^T\|\varepsilon^{-1/2}f(t)\|_{L^2(\Omega)}^2\,dt.
\]
After using H\"older's inequality, it is easy to see the constant $C$ depends only on $\|\varepsilon^{-1}\|_{L^\infty}$, which completes the proof.
\end{proof}

\begin{lemma}\label{lem:unifboundEBpl}
Let $(E_h^{n},B_h^{n})$ solve the CN scheme \eqref{eq:CN}, and let $E_h^{\mathrm{pl}}$, $B_h^{\mathrm{pl}}$, $E_h^{1/2}$, $B_h^{1/2}$ be the lifts defined above. Moreover, let $f \in L^2(0,T;L^2(\Omega)^3)$, then 
\[
\|E_h^{\mathrm{pl}}\|_{L^\infty(0,T;L^2(\Omega))}+\|B_h^{\mathrm{pl}}\|_{L^\infty(0,T;L^2(\Omega))}\le C,
\qquad
\|E_h^{1/2}\|_{L^2(0,T;L^2(\Omega))}+\|B_h^{1/2}\|_{L^2(0,T;L^2(\Omega))}\le C,
\]
where the constant $C > 0$ is independent of $h$ and $\Delta t$. 
\end{lemma}
\begin{proof}
By Lemma~\ref{lem:CN-wp} (energy identity and stability) and Lemma~\ref{lem:fbound} we have
\[
\mathcal E^m+\Delta t\sum_{n=0}^{m-1}\|\sqrt{\sigma}\,E_h^{n+1/2}\|_{L^2(\Omega)}^2
\;\le\;
C_T\Big(\mathcal E^0+\sum_{n=0}^{m-1}\Delta t\,\|\varepsilon^{-1/2}f^{n+1/2}\|_{L^2(\Omega)}^2\Big)
\;\le\;C \quad \forall~ m,
\]
with $C$ independent of $h,\Delta t$. 
By uniform ellipticity of $\varepsilon$ and $\mu^{-1}$, the discrete energy $\mathcal{E}^m$ controls $\|E_h^m\|_{L^2(\Omega)}^2 + \|B_h^m\|_{L^2(\Omega)}^2.$ Hence, for every node $t^m$,
\[
\|E_h^m\|_{L^2(\Omega)}^2 + \|B_h^m\|_{L^2(\Omega)}^2 \le C,
\]
so that $\sup_m \|E_h^m\|_{L^2(\Omega)}+\sup_m \|B_h^m\|_{L^2(\Omega)}\le C$ (with a possibly larger $C$).

Now, on each interval $(t^n,t^{n+1}]$, note that the piecewise linear reconstructions satisfy
\[
E_h^{\mathrm{pl}}(t)=(1-\theta)E_h^n+\theta E_h^{n+1},\qquad 
B_h^{\mathrm{pl}}(t)=(1-\theta)B_h^n+\theta B_h^{n+1},
\quad \theta=\frac{t-t^n}{\Delta t}\in(0,1].
\]
By convexity of the norm,
\[
\|E_h^{\mathrm{pl}}(t)\|_{L^2(\Omega)}
\le (1-\theta)\|E_h^n\|_{L^2(\Omega)}+\theta\|E_h^{n+1}\|_{L^2(\Omega)}
\le \max\{\|E_h^n\|_{L^2(\Omega)},\|E_h^{n+1}\|_{L^2(\Omega)}\},
\]
and similarly for $B_h^{\mathrm{pl}}$. Taking the supremum in $t$ gives
\[
\|E_h^{\mathrm{pl}}\|_{L^\infty(0,T;L^2(\Omega))} \le \sup_{0\le m\le N_t}\|E_h^m\|_{L^2(\Omega)}\le C,
\qquad
\|B_h^{\mathrm{pl}}\|_{L^\infty(0,T;L^2(\Omega))} \le \sup_{0\le m\le N_t}\|B_h^m\|_{L^2(\Omega)}\le C.
\]
For the midpoint lift $E_h^{1/2}(t)=E_h^{n+1/2}:=\tfrac12(E_h^{n+1}+E_h^n)$ on $(t^n,t^{n+1}]$,
\begin{multline*}
\|E_h^{1/2}\|_{L^2(0,T;L^2(\Omega))}^2
=\sum_{n=0}^{N_t-1}\Delta t\,\big\|E_h^{n+1/2}\big\|_{L^2(\Omega)}^2
\le \sum_{n=0}^{N_t-1}\Delta t\,\frac{\|E_h^{n+1}\|_{L^2(\Omega)}^2+\|E_h^{n}\|_{L^2(\Omega)}^2}{2}\\
\le T\,\sup_{0\le m\le N_t}\|E_h^m\|_{L^2(\Omega)}^2 \le C.
\end{multline*}
Thus,
\[
\|E_h^{\mathrm{pl}}\|_{L^\infty(0,T;L^2(\Omega))}+\|B_h^{\mathrm{pl}}\|_{L^\infty(0,T;L^2(\Omega))}\le C,
\qquad
\|E_h^{1/2}\|_{L^2(0,T;L^2(\Omega))} + \|B_h^{1/2}\|_{L^2(0,T;L^2(\Omega))}\le C.
\]
The proof is complete.
\end{proof}

\begin{lemma}[Midpoint and piecewise–linear reconstructions have the same limit]
\label{lem:pl-vs-midpoint}
Let $(E_h^{n},B_h^{n})$ solve the CN scheme \eqref{eq:CN}, and let
$E_h^{\mathrm{pl}}, B_h^{\mathrm{pl}}, E_h^{1/2}, B_h^{1/2}$ be the lifts defined above.
Let $f \in L^2(0,T;L^2(\Omega)^3)$. 
Then there exist functions $ E, B\in L^\infty(0,T;L^2(\Omega)^3)$ and a subsequence (not relabeled) such that
\[
E_h^{\mathrm{pl}}\stackrel{*}{\rightharpoonup}\, E,\qquad
B_h^{\mathrm{pl}}\stackrel{*}{\rightharpoonup}\, B
\quad\text{in }L^\infty(0,T;L^2(\Omega)^3).
\]
Moreover, the midpoint reconstructions converge to the same limits in the following sense:
for every $\Phi\in L^1(0,T;L^2(\Omega)^3)$,
\[
\int_0^T \!\big(E_h^{1/2}(t),\Phi(t)\big)\,dt \;\longrightarrow\; 
\int_0^T \!\big( E(t),\Phi(t)\big)\,dt,
\quad
\int_0^T \!\big(B_h^{1/2}(t),\Phi(t)\big)\,dt \;\longrightarrow\; 
\int_0^T \!\big( B(t),\Phi(t)\big)\,dt.
\]
In other words, $E_h^{1/2}\to E$ and $B_h^{1/2}\to B$ in the duality with
$L^1(0,T;L^2(\Omega)^3)$.
\end{lemma}

\begin{proof}
We detail the argument for $B$; the proof for $E$ is identical.

\smallskip\noindent\emph{Step 1 (choose a subsequence for the piecewise–linear lifts).}
By Lemma~\ref{lem:unifboundEBpl}, $\{B_h^{\mathrm{pl}}\}$ is bounded in
$L^\infty(0,T;L^2(\Omega)^3)$. By the Banach–Alaoglu theorem there exists
$ B\in L^\infty(0,T;L^2(\Omega)^3)$ and a subsequence (not relabeled) such that
$B_h^{\mathrm{pl}}\stackrel{*}{\rightharpoonup} B$ in $L^\infty(0,T;L^2(\Omega)^3)$.

\smallskip\noindent\emph{Step 2 (midpoint equals the cell average of the piecewise–linear lift).}
On each interval $I_n=(t^n,t^{n+1}]$, $B_h^{\mathrm{pl}}$ is affine in $t$, hence
\[
\frac{1}{\Delta t}\int_{t^n}^{t^{n+1}} B_h^{\mathrm{pl}}(s)\,ds
= \frac{B_h^{n+1}+B_h^n}{2}=B_h^{n+1/2}.
\]
Define $(\mathcal M_{\Delta t}v)(t):=\frac{1}{\Delta t}\int_{t^n}^{t^{n+1}}v(s)\,ds$ for $t\in I_n$.
Then
\begin{equation}\label{eq:Mh-id}
B_h^{1/2}=\mathcal M_{\Delta t}\big(B_h^{\mathrm{pl}}\big)\quad\text{on }(0,T].
\end{equation}

\smallskip\noindent\emph{Step 3 (duality identity with averaged test functions).}
Let $\Phi\in L^1(0,T;L^2(\Omega)^3)$ and set
\(
(\Pi_{\Delta t}\Phi)(s):=\frac{1}{\Delta t}\int_{t^n}^{t^{n+1}}\Phi(t)\,dt
\) for $s\in I_n$.
Using \eqref{eq:Mh-id} and Fubini's theorem (which holds since $B_h^{\mathrm{pl}}\in L^\infty(0,T; L^2(\Omega))$ and
$\Phi\in L^1(0,T;L^2(\Omega))$), we obtain
\[
\int_0^T \big(B_h^{1/2}(t),\Phi(t)\big)\,dt
=\sum_{n=0}^{N_t-1}\int_{I_n}\Big(\frac{1}{\Delta t}\int_{I_n} B_h^{\mathrm{pl}}(s)\,ds,\ \Phi(t)\Big)\,dt
=\int_0^T \big(B_h^{\mathrm{pl}}(s),(\Pi_{\Delta t}\Phi)(s)\big)\,ds.
\]

\smallskip\noindent\emph{Step 4 (pass to the limit against $L^1$ test functions).}
It is standard that $\Pi_{\Delta t}\Phi\to\Phi$ in $L^1(0,T;L^2(\Omega)^3)$ as $\Delta t\to0$, see  Lemma~\ref{lem:time-average-convergence}.
Using $B_h^{\mathrm{pl}}\stackrel{*}{\rightharpoonup} B$ in $L^\infty(0,T;L^2(\Omega))$, we get $\forall\,\Phi\in L^1(0,T;L^2(\Omega)^3)$
\[
\int_0^T \big(B_h^{1/2}(t),\Phi(t) \big)\,dt
=\int_0^T \big(B_h^{\mathrm{pl}}(s),\Pi_{\Delta t}\Phi(s)\big)\,dt
\;\longrightarrow\; \int_0^T \big( B(t),\Phi(t)\big) \,dt.
\]
This proves the asserted convergence in the duality with $L^1(0,T;L^2(\Omega))$.
\end{proof}

\begin{theorem}[Convergence of the fully discrete CN scheme]\label{thm:FD-convergence}
Let $(E_h^{n},B_h^{n})_{n=0}^{N_t}$ solve the Crank--Nicolson scheme \eqref{eq:CN}, and let
$E_h^{\mathrm{pl}},B_h^{\mathrm{pl}},E_h^{1/2},B_h^{1/2}$ be the piecewise linear / midpoint lifts.
Let $f \in L^2(0,T;L^2(\Omega)^3)$. Then the limit from Lemma~\ref{lem:pl-vs-midpoint}
is the unique solution to \eqref{eq:MaxWeak} in the sense of Definition~\ref{def:weak_soln}. 
\end{theorem}

\begin{proof}
\emph{Step 1: Compactness.}
By Lemma~\ref{lem:pl-vs-midpoint}, we have that up to subsequence 
\[
E_h^{\mathrm{pl}}\stackrel{*}{\rightharpoonup} E,\qquad
B_h^{\mathrm{pl}}\stackrel{*}{\rightharpoonup} B
\quad\text{in }L^\infty(0,T;L^2(\Omega)^3),
\]
and, moreover,
\[
E_h^{1/2}\rightharpoonup  E,\qquad
B_h^{1/2}\rightharpoonup  B
\quad\text{in duality with } L^1(0,T;L^2(\Omega)^3).
\]

\medskip
\emph{Step 2: Discrete integration by parts in time (with smooth cutoff).}
Let $\eta\in C_c^\infty(0,T)$ and set $\eta^{n+1/2}:=\frac{1}{\Delta t}\int_{t^n}^{t^{n+1}}\eta(t)\,dt$.
For any time-independent discrete $v_h\in L^2(\Omega)^3$,
\begin{equation}\label{eq:disc-ibp0}
\sum_{n=0}^{N_t-1}\Delta t\,(\delta_t u^{n+1},v_h)\,\eta^{n+1/2}
\;=\;-\int_0^T (u_h^{\mathrm{pl}}(t),v_h)\,\eta'(t)\,dt,
\end{equation}
because $u_h^{\mathrm{pl}}$ is affine on each $I_n$ and $\eta^{n+1/2}$ is the cell average.
See Lemma~\ref{lem:disc-IBP} for more details regarding Equation~\eqref{eq:disc-ibp0}. 

\medskip
\emph{Step 3: Limit equation for Faraday's law.}
Fix $\Phi\in H(\curl;\Omega)$ and let $\phi_h:=P_h\Phi\in \mathcal{RT}_h$, where $P_h:L^2(\Omega)^3\to \mathcal{RT}_h$
is the $L^2$-stable RT projector (see Proposition~\ref{prop:L2proj-RT}) with $P_h\Phi\to\Phi$ in $L^2$.
Multiply \eqref{eq:CN_faraday} by $\eta^{n+1/2}\Delta t$ and sum over $n$:
\[
\sum_{n} \Delta t\,(\delta_t B_h^{n+1},\phi_h)\,\eta^{n+1/2}
\;+\;\sum_{n}\Delta t\,(\curl E_h^{n+1/2},\phi_h)\,\eta^{n+1/2}=0.
\]
Use \eqref{eq:disc-ibp0} for the first sum and the midpoint lift for the second:
\[
-\int_0^T (B_h^{\mathrm{pl}}(t),\phi_h)\,\eta'(t)\,dt
\;+\;\int_0^T (\curl E_h^{1/2}(t),\phi_h)\,\eta(t)\,dt=0.
\]
Since $\curl E_h^{1/2}(t)\in \mathcal{RT}_h$ and $P_h$ is the $L^2$ projector onto $\mathcal{RT}_h$, we have
$(\curl E_h^{1/2},\phi_h)=(\curl E_h^{1/2},\Phi)$. Next we will pass through the limit $h,\Delta t\to0$, term by term.

\smallskip\noindent\emph{Term 1: $-\displaystyle\int_0^T(B_h^{\mathrm{pl}},\phi_h)\eta'$}.
Add and subtract $\Phi$:
\[
\int_0^T (B_h^{\mathrm{pl}},\phi_h)\,\eta'
=\int_0^T \big(B_h^{\mathrm{pl}},\Phi\big)\,\eta'
+\int_0^T \big(B_h^{\mathrm{pl}},\phi_h-\Phi\big)\,\eta'=:I_{1,h}+I_{2,h}.
\]
Since $B_h^{\mathrm{pl}}\stackrel{*}{\rightharpoonup} B$ in $L^\infty(0,T;L^2(\Omega))$, due to
Lemma~\ref{lem:pl-vs-midpoint}, and
$\eta'\,\Phi\in L^1(0,T;L^2(\Omega))$, we have
\[
I_{1,h}\;\longrightarrow\;\int_0^T \big( B,\Phi\big)\,\eta'.
\]
For the remainder, use the uniform bound $\|B_h^{\mathrm{pl}}\|_{L^\infty(0,T;L^2(\Omega))}\le C$ and
$P_h\Phi\to\Phi$ in $L^2(\Omega)^3$:
\[
|I_{2,h}|
\le \|B_h^{\mathrm{pl}}\|_{L^\infty(0,T;L^2(\Omega))}\,\|\eta'\|_{L^1(0,T)}\,\|\phi_h-\Phi\|_{L^2(\Omega)}
\;\xrightarrow[h\to0]{}\;0.
\]
Hence
\[
-\int_0^T (B^{\mathrm{pl}}_h,\phi_h)\,\eta'
\;\longrightarrow\;
-\int_0^T ( B,\Phi)\,\eta'.
\]

\smallskip\noindent\emph{Term 2: $\displaystyle\int_0^T (E_h^{1/2},\curl\Phi)\,\eta$}.
Since $\eta\,\curl\Phi\in L^1(0,T;L^2(\Omega)^3)$ and $E_h^{1/2}\rightharpoonup  E$ (due to 
Lemma~\ref{lem:pl-vs-midpoint}) in the
duality with $L^1(0,T;L^2(\Omega))$, we directly obtain
\[
\int_0^T \big(E_h^{1/2},\curl\Phi\big)\,\eta
\;\longrightarrow\;
\int_0^T \big( E,\curl\Phi\big)\,\eta.
\]

Combining the two limits yields
\[
\begin{aligned}
-\int_0^T \big( B,\Phi\big)\,\eta'
&=- \int_0^T \big( E,\curl\Phi\big)\,\eta , 
\qquad \forall \Phi \in  H(\mathrm{curl};\Omega) \, , \forall \eta \in C_c^\infty(0,T) .
\end{aligned}
\]
The remaining steps involving identification of the time derivative via distributional testing, extension by density of the test space, and a standard uniqueness argument to upgrade subsequence convergence to full convergence, closely mimic the proof of~\cite[Theorem~3]{HAntil_2025a} and are omitted for brevity.

\medskip
\emph{Step 4: Limit equation for Amp\`ere's law.}
Fix $\Psi\in H_0(\curl;\Omega)$ and set $\psi_h:=R_h\Psi\in \mathcal N_h^0$, where
$R_h:H_0(\curl)\to \mathcal N_h^0$ is the $L^2$-stable N\'ed\'elec projector
(Proposition~\ref{prop:Riesz-Hcurl}) with $R_h\Psi\to\Psi$ in $L^2(\Omega)^3$ and
$\curl R_h\Psi\to\curl\Psi$ in $L^2(\Omega)^3$.
Multiply \eqref{eq:CN_ampere} by $\eta^{n+1/2}\Delta t$ and sum over $n$:
\[
\sum_{n}\Delta t\,\langle \varepsilon\,\delta_t E_h^{n+1},\psi_h\rangle\,\eta^{n+1/2}
+ \int_0^T (\sigma E_h^{1/2},\psi_h)\,\eta\,dt
- \int_0^T (\mu^{-1}B_h^{1/2},\curl\psi_h)\,\eta\,dt
= \int_0^T (f^{1/2},\psi_h)\,\eta\,dt.
\]
Apply \eqref{eq:disc-ibp0} to the first sum (with $u=E_h$ and test $\varepsilon\psi_h$) to get
\[
-\int_0^T (\varepsilon E_h^{\mathrm{pl}},\psi_h)\,\eta'(t)\,dt
+ \int_0^T \!\Big[(\sigma E_h^{1/2},\psi_h)
-(\mu^{-1}B_h^{1/2},\curl\psi_h)\Big]\eta\,dt
= \int_0^T (f^{1/2},\psi_h)\,\eta\,dt.
\]
Remaining steps, after using the properties of $R_h$ from Proposition~\ref{prop:Riesz-Hcurl}, 
closely mimic \cite[Theorem~3]{HAntil_2025a} and are omitted for brevity.
\end{proof}

\paragraph{Convergence of the discrete Gauss laws.}
\begin{lemma}[Strong convergence of $\divv B_h^n$]
\label{lem:divBhn-strong-zero-upd}
Let $(E_h^n,B_h^n)$ solve \eqref{eq:CN} and let the assumptions of 
Lemma~\ref{lem:disc_gauss} and Theorem~\ref{thm:FD-convergence} hold.
If moreover
\[
(\divv B_h^0,q_h)=0\qquad \forall q_h\in\mathcal Q_h,
\]
then
\[
\divv B_h^{\mathrm{pl}} \equiv 0 \quad \text{in } L^\infty(0,T;L^2(\Omega)).
\]
In particular,
\[
\divv B_h^{\mathrm{pl}}\to 0 \quad \text{strongly in } L^\infty(0,T;L^2(\Omega)).
\]
and
\[
\divv B = 0
\]
in the sense of distributions on \(\Omega\times(0,T)\).
\end{lemma}

\begin{proof}
For all $n$, since \(B_h^n\in\mathcal{RT}_h\), we have
$\divv B_h^n\in\mathcal Q_h.$
By Lemma~\ref{lem:disc_gauss} we have that
\[
\delta_t(\mathrm{div}_h B_h^{n+1})=0
\]
By the definition of the piecewise linear reconstruction $B_h^{\mathrm{pl}}$ this is equivalent to 
\[
\partial_t(\divv B_h^{\mathrm{pl}}(t),q_h)=0\qquad \forall q_h\in\mathcal Q_h \quad \text{ for } t\in (t^n, t^{n+1}].
\]
and using the initial data then implies
\[
(\divv B_h^{\mathrm{pl}},q_h)=0\qquad \forall q_h\in\mathcal Q_h \qquad\text{for all } n.
\]
The rest of the proof follows similarly as in \cite[Lemma 5]{HAntil_2025a} and is omitted for brevity. 
\end{proof}

\begin{lemma}[Convergence of the discrete electric Gauss law]
\label{lem:fully-disc-gaussE-conv}
Let the assumptions of Lemma~\ref{lem:fully_disc_gauss_E} and Theorem~\ref{thm:FD-convergence} hold and $I_h : H_0^1(\Omega)\to \mathcal P_h$ be an interpolation operator and define 
\[
\langle \divv_h(v), q_h \rangle_{\mathcal P_h^*,\mathcal P_h}
:=-(v,\nabla q_h)
\qquad \forall q_h\in \mathcal P_h.
\]
Then for every $\phi\in C_c^\infty(\Omega)$ and $\eta\in C_c^\infty(0,T)$,
\[
\int_0^T \eta'(t)\,\langle \divv_h(\varepsilon E_h^{\mathrm{pl}}), I_h\phi\rangle\,dt
-
\int_0^T \eta(t)\,\langle \divv_h(\sigma E_h^{1/2}(t)-f^{1/2}(t)), I_h\phi\rangle\,dt
=0.
\]
Passing to the limit $h\to0$ and $\Delta t \to0$ yields
\[
\partial_t \divv(\varepsilon E)+\divv(\sigma E-f)=0
\qquad\text{in }\mathcal D'(\Omega\times(0,T)).
\]
\end{lemma}
\begin{proof} 
Let $\phi\in C_c^\infty(\Omega)$. Recalling that $I_h\phi \to \phi \quad \text{in } H_0^1(\Omega),
\quad\text{we have that }\quad \nabla I_h\phi \to \nabla\phi \text{ in }L^2(\Omega)^3.$ And since $\nabla \mathcal P_h\subset \mathcal N_h^0$, we can choose \(\psi_h=\nabla I_h\phi \)
as a test function in \eqref{eq:CN_ampere}. 
Multiplying \eqref{eq:CN_ampere} by $\eta^{n+1/2}\Delta t$, summing over $n$ and applying \eqref{eq:disc-ibp0} to the first sum (with $u=E_h$ and test $\psi_h=\nabla I_h\phi$)
we obtain, for $t$ a.e.\ in $(0,T)$, (because
\(\curl(\nabla I_h\phi)=0,\)):
\[
-\int_0^T (\varepsilon E_h^{\mathrm{pl}},\nabla I_h\phi)\,\eta'(t)\,dt
+ \int_0^T \!\Big[(\sigma E_h^{1/2} - f^{1/2},\nabla I_h\phi)\Big]\eta\,dt
= 0.
\]
The rest of the proof follows along similar lines as in \cite[Lemma 6]{HAntil_2025a} with $E_h^{\mathrm{pl}}$ instead of $E_h$ and is omitted for brevity.
\end{proof}

\section{Optimal control of the time-dependent Maxwell's equations}
\label{sec:ocp}

Next we analyze the optimal control problem.
Throughout the remainder, we adopt Assumption~\ref{ass:data} on the problem coefficients.
Additionally, to simplify the notation, we consider
$
 \Omega_\text{src} = \Omega_\text{obs} = \Omega_\text{ctrl} = \Omega.
$
Our analysis remains valid when $\Omega_\text{src}$, $\Omega_\text{obs}$,  and $\Omega_\text{ctrl}$ are proper subsets of~$\Omega$.

\subsection{State equation and the control-to-state map}
Given a control $z\in Z$, the state $(E,B)$ is defined as the solution to \eqref{eq:state-ocp} according to Definition~\ref{def:weak_soln}. We state the well-posedness next.

\begin{lemma}[Well-posedness and bounded control-to-state map]\label{lem:S-bounded}
Under Assumption~\ref{ass:data}, for every $z\in Z$ there exists a unique weak solution 
$(E,B)$ of \eqref{eq:state-ocp} in the sense of Definition~\ref{def:weak_soln}
(with $f := I_{\rm src}+\mathrm{curl}\,z$). Moreover,
\[
\begin{aligned}
\| \partial_t E \|_{L^2(0,T;H_0(\mathrm{curl};\Omega)^*)}  
&+ \| \partial_t B \|_{L^2(0,T;H(\curl; \Omega)^*)} 
+  \|E\|_{C([0,T];L^2(\Omega)^3)} + \|B\|_{C([0,T];L^2(\Omega))} \\
&\;\le\; C\Big( \|I_{\rm src}\|_{L^2(0,T;L^2(\Omega))} + \|z\|_{Z} + \|E_0\|_{L^2(\Omega)}+\|B_0\|_{L^2(\Omega)}\Big),
\end{aligned}
\]
with a constant $C>0$ independent of $z$. Consequently, the control-to-state map
\[
S:Z\to Y,\qquad S(z)=(E,B),
\]
is \emph{affine} and bounded. Writing $S(z)=S_0 z + (E^{\rm src},B^{\rm src})$, the operator
$S_0:Z\to Y$ is linear and bounded. Here $S_0$ corresponds to $I_{\rm src} \equiv 0$.  
\end{lemma}
\begin{proof}
Set $f:=I_{\rm src}+\mathrm{curl}\,z \in L^2(0,T;L^2(\Omega)^3)$ for $z\in Z$.
The a priori estimate follows from Theorem~\ref{thm:main} after using 
$\|f\|_{L^2(0,T;L^2(\Omega)^3)} \le \|I_{\rm src}\|_{L^2(0,T;L^2(\Omega)^3)} + \|\mathrm{curl}\,z\|_{L^2(0,T;L^2(\Omega)^3)}$, we immediately obtain the stated bound and boundedness of $S$.
Linearity of $S_0$ follows from the linearity of the state system in the forcing $\mathrm{curl}\,z$.
\end{proof}

\begin{remark}
\rm 
    To incorporate $\varepsilon$ and $\mu^{-1}$ in the objective function $J$ in \eqref{eq:intro_obj}, we equip $Y$ with the weighted scalar product 
    \[
    ((E,B), e,b))_Y := (\varepsilon E, e) + (\mu^{-1} B, b) \qquad \forall (E,B), (e,b) \in Y,
    \]
    which has the corresponding weighted norm
    \[
    \|(E,B)\|_Y := \sqrt{ (\varepsilon E, E) + (\mu^{-1} B, B)} \qquad \forall (E,B) \in Y.
    \]    
    Due to the assumptions on $\varepsilon$ and $\mu$, 
    $(\cdot, \cdot)_Y$ indeed defines an inner product on $Y$ and the bound in Lemma~\ref{lem:S-bounded} also holds using this weighted norm.
\end{remark}

\subsection{The reduced optimal control problem}
Having defined the control-to-state map $S(z)$, we formulate a reduced optimal control problem.
Given the target fields $(E_d,B_d)\in Y$ and $\alpha_1,\alpha_2>0$, we consider 
\begin{equation}\label{eq:reduced}
\min_{z\in Z}\;
\left\{ \mathcal{J}(z)
:= \frac12 \| S(z) - (E_d,B_d) \|_Y^2
 + \frac{\alpha_1}{2} \| z \|_{L^2(0,T;L^2(\Omega))}^2
 + \frac{\alpha_2}{2} \| \mathrm{curl}\,z \|_{L^2(0,T;L^2(\Omega))}^2 \right\} .
\end{equation}
  
\begin{theorem}[Existence and uniqueness of an optimal control]\label{thm:exist-unique-ocp}
Under Assumption~\ref{ass:data}, problem \eqref{eq:reduced} admits a unique minimizer $z^*\in Z$.
\end{theorem}

\begin{proof}
The result follows immediately from the direct method in the calculus of variations.
Let $\{z_n\}\subset Z$ be a minimizing sequence. Set $(E_n,B_n):=S(z_n)$.
By Lemma~\ref{lem:S-bounded} and the definition of $\mathcal{J}$,
\[
\lim_{n \rightarrow \infty} \mathcal{J}(z_n) \;=\; \inf_{z} \mathcal{J}(z) \le \mathcal{J}(0) <\infty,
\]
hence (up to a subsequence) $z_n \rightharpoonup z^*$ in $Z$. By boundedness and linearity
of $S_0$, we have $S_0 z_n \rightharpoonup S_0 z^*$ in $Y$; since $S$ is affine,
$S(z_n) \rightharpoonup S(z^*)$ in $Y$.

The functional $\mathcal{J}$ is the sum of: (i) a convex continuous quadratic in $Y$ composed
with the affine map $S$, hence weakly lower semicontinuous in $z$; (ii) the strictly convex,
coercive quadratic form
\[
\frac{\alpha_1}{2}\|z\|_{L^2(0,T;L^2(\Omega))}^2+\frac{\alpha_2}{2}\|\mathrm{curl}\,z\|_{L^2(0,T;L^2(\Omega))}^2,
\]
which is weakly lower semicontinuous on $Z$. Therefore,
\[
\mathcal{J}(z^*) \;\le\; \liminf_{n\to\infty} \mathcal{J}(z_n)
 = \inf_{z\in Z} \mathcal{J}(z).
\]
Thus $z^*$ is a minimizer. Uniqueness follows from strict convexity of $\mathcal{J}$ on $Z$,
which in turn is guaranteed by $\alpha_1, \alpha_2>0$.
\end{proof}

\subsection{Adjoint and characterization of $S_0^*$}
\label{sec:adjoint}
We now characterize the Hilbert adjoint $S_0^*:Y\to Z^*$ without assuming any extra spatial regularity on the adjoint variables. Let $(q_E,q_B)\in Y$ be arbitrary. For each $w\in Z$, let $(\delta E,\delta B)=S_0 w$ denote the solution of \eqref{eq:state-ocp} with $z$ replaced by $w$ and homogeneous initial data.

\begin{definition}[Weak solution to adjoint]
\label{def:adjoint_veryweak}
A pair $(p_E,p_B)$ is called an adjoint for $(q_E,q_B)\in Y$ if
\[
\begin{aligned}
p_E &\in H^1\!\big(0,T;H_0(\mathrm{curl};\Omega)^*\big)\cap C\big([0,T];L^2(\Omega)^3\big), \\
p_B &\in H^1\!\big(0,T;H(\curl;\Omega)^*\big)\cap C\big([0,T];L^2(\Omega)^3\big) ,
\end{aligned}
\]
$p_E(T)=0$, $p_B(T)=0$ in $L^2(\Omega)^3$ such that
\begin{equation}\label{eq:adjoint_identity}
\begin{aligned}
\langle -\varepsilon \partial_t \, p_E , \, \psi \rangle_{H_0(\curl;\Omega)^*,H_0(\curl;\Omega)} + (p_E , \sigma \psi) 
+( \mathrm{curl}\,\psi,\,p_B )
&=  (q_E,\psi) \quad \forall \psi \in H_0(\mathrm{curl};\Omega)) \\
\langle -\partial_t p_B,\, \phi\rangle_{H(\curl;\Omega)^*,H(\curl;\Omega)}
- ( \mu^{-1}\phi,\, \mathrm{curl}\, p_E ) 
&= (q_B,\phi) \quad \forall \phi \in H(\curl;\Omega).
\end{aligned}
\end{equation}
\end{definition}
Notice that existence and uniqueness of solution to the adjoint equation \eqref{eq:adjoint_identity}
follows by the same argument as in Theorem~\ref{thm:main}. 

\begin{proposition}[Characterization of $S_0^*$]
\label{prop:S0star}
Let
$S_0:Z\to Y$ be the linear control-to-state map that assigns to
$w\in Z$ the unique weak solution $(\delta E,\delta B)=S_0 w$ of
\[
\begin{cases}
\varepsilon\,\partial_t \delta E - \mathrm{curl}\,(\mu^{-1}\delta B)+\sigma \delta E=\mathrm{curl}\,w
&\text{in }H_0(\mathrm{curl};\Omega)^* \mbox{ a.e. } t \in (0,T) ,\\[2pt]
\partial_t \delta B+\mathrm{curl}\,\delta E=0
&\text{in }H(\curl;\Omega)^* \mbox{ a.e. } t \in (0,T),\\[2pt]
\delta E(0)=0,\quad \delta B(0)=0 \quad \mbox{in } L^2(\Omega)^3 .
\end{cases}
\]
For any $(q_E,q_B)\in Y$, let $(p_E,p_B)$ be the (weak) adjoint 
from Definition~\ref{def:adjoint_veryweak} with terminal conditions $p_E(T)=0$, $p_B(T)=0$. 
Then, for every $w\in Z$ with $(\delta E,\delta B)=S_0 w$,
\begin{equation}\label{eq:S0star_identity}
\int_0^T\!\big[(q_E,\delta E)+(q_B,\delta B)\big]\,dt
=\int_0^T ( p_E,\,\mathrm{curl}\,w )\,dt.
\end{equation}
Equivalently,
\[
S_0^*(q_E,q_B)=\mathcal{C}^* p_E \in Z^*, \qquad
\text{where }\ \mathcal{C}:Z\to L^2(0,T;L^2(\Omega)^3),\ \ \mathcal{C}z:=\mathrm{curl}\,z,
\]
and $\mathcal{C}^*:L^2(0,T;L^2(\Omega)^3)\to Z^*$ is given by
\(\ \langle \mathcal{C}^*v, z\rangle_{Z^*,Z} = \int_0^T (v,\mathrm{curl}\,z)\) dt.
\end{proposition}

\begin{proof}
\emph{Step 1 (Galerkin spaces).}
Let $\{Y_N\}_{N\in\mathbb{N}}\subset H_0(\mathrm{curl};\Omega)$
and $\{X_N\}_{N\in\mathbb{N}}\subset H(\mathrm{curl};\Omega)$ be increasing, finite-dimensional
subspaces such that $\overline{\bigcup_N Y_N}^{\,L^2}=L^2(\Omega)^3$
and $\overline{\bigcup_N X_N}^{\,L^2}=L^2(\Omega)^3$.
We assume the corresponding $L^2$-orthogonal projections
$\Pi_N':L^2(\Omega)^3\to Y_N$ and $\Pi_N:L^2(\Omega)^3\to X_N$
satisfy $\Pi_N' v\to v$ and $\Pi_N v\to v$ in $L^2(\Omega)^3$ for every $v\in L^2(\Omega)^3$.

\smallskip
\emph{Step 2 (State Galerkin problem).}
Given $w\in Z$, define $(\delta E_N,\delta B_N)\in L^2(0,T;Y_N)\times L^2(0,T;X_N)$ with
$\delta E_N(0)=0$, $\delta B_N(0)=0$, as the unique solution of
\begin{equation}\label{eq:state_disc}
\begin{aligned}
\langle \varepsilon \partial_t \delta E_N, \psi\rangle + (\sigma \delta E_N,\psi)
- (\mu^{-1}\delta B_N, \mathrm{curl}\,\psi) &= (\mathrm{curl}\,w,\psi) &&\forall \psi\in Y_N,\\
\langle \partial_t \delta B_N,\phi\rangle + (\mathrm{curl}\,\delta E_N,\phi) &= 0 &&\forall \phi\in X_N.
\end{aligned}
\end{equation}
By standard ODE theory in finite dimensions and the a priori energy estimate already proved for the forward problem in \cite[Theorem~2]{HAntil_2025a}, we have the uniform bounds
\[
\delta E_N \ \text{bounded in } L^\infty(0,T;L^2(\Omega)),\qquad
\delta B_N \ \text{bounded in } L^\infty(0,T;L^2(\Omega)).
\]
Moreover (up to subsequences), 
\[
\delta E_N \overset{*}{\rightharpoonup} \delta E \text{ in } L^\infty(0,T;L^2(\Omega)),\quad
\delta B_N \overset{*}{\rightharpoonup} \delta B \text{ in } L^\infty(0,T;L^2(\Omega)),
\]
and, shown in \cite[Theorem~2]{HAntil_2025a}, $(\delta E,\delta B)$ is the 
unique weak solution $S_0 w$.

\smallskip
\emph{Step 3 (Adjoint Galerkin problem).}
For the given $(q_E,q_B)\in Y$, define $(p_E^N,p_B^N)\in L^2(0,T;Y_N)\times L^2(0,T;X_N)$ with
terminal conditions $p_E^N(T)=0$, $p_B^N(T)=0$ by
\begin{equation}\label{eq:adjoint_disc}
\begin{aligned}
\langle -\varepsilon \partial_t p_E^N,\psi\rangle + (\sigma p_E^N,\psi) 
+ (\mathrm{curl}\,\psi,\,p_B^N) &= (q_E,\psi) &&\forall \psi\in Y_N,\\
\langle -\partial_t p_B^N,\phi\rangle - ( \mu^{-1}\phi,\,\mathrm{curl}\, p_E^N ) &= (q_B,\phi) &&\forall \phi\in X_N.
\end{aligned}
\end{equation}
Again, finite dimensional ODE theory and the adjoint energy estimate yield the uniform bounds
\[
p_E^N,\,p_B^N \ \text{bounded in } L^\infty(0,T;L^2(\Omega)).  
\]
The proof goes along the lines of the forward problem as shown in \cite[Theorem~2]{HAntil_2025a}. 
Consequently (up to subsequences),
\[
p_E^N \overset{*}{\rightharpoonup} p_E \text{ in } L^\infty(0,T;L^2(\Omega)),\qquad
p_B^N \overset{*}{\rightharpoonup} p_B \text{ in } L^\infty(0,T;L^2(\Omega)),
\]
and the limit $(p_E,p_B)$ is the unique weak adjoint from 
Definition~\ref{def:adjoint_veryweak}.

\emph{Step 4 (Discrete duality identity).}
Fix $N\in\mathbb{N}$. 
Take in the discrete adjoint system~\eqref{eq:adjoint_disc} the test functions
$\psi=\delta E_N(t)\in Y_N \subset H_0(\mathrm{curl};\Omega)$ and
$\phi=\delta B_N(t)\in X_N \subset H(\curl;\Omega)$.
Integrating over~$(0,T)$ and integrating by parts in time
(using $p_E^N(T)=p_B^N(T)=0$ and $\delta E_N(0)=\delta B_N(0)=0$), we get
\begin{align*}
\int_0^T\!\!\big[(q_E,\delta E_N)+(q_B,\delta B_N)\big]\,dt
&=\int_0^T\!\!\Big(
\langle \varepsilon \partial_t \delta E_N, p_E^N\rangle 
+ (\sigma \delta E_N,p_E^N)
+ (\mathrm{curl}\,\delta E_N,\,p_B^N) \\
&\hspace{7.0em} 
+ \langle \partial_t \delta B_N, p_B^N\rangle
- (\mu^{-1} \delta B_N,\, \mathrm{curl}\,  p_E^N)\Big)\,dt. \tag{$\star$}
\end{align*}

Next, take in the discrete state increment system \eqref{eq:state_disc} the test
functions $\psi=p_E^N(t)\in Y_N\subset H_0(\mathrm{curl};\Omega)$ and 
$\phi=p_B^N(t)\in X_N\subset H(\mathrm{curl};\Omega)$, and integrate over $(0,T)$:
\begin{align*}
\int_0^T\!\!\Big(
\langle \varepsilon \partial_t \delta E_N, p_E^N\rangle 
+ (\sigma \delta E_N,p_E^N)
- (\mu^{-1}\delta B_N, \mathrm{curl}\,p_E^N)
+ \langle \partial_t \delta B_N,  p_B^N\rangle 
+ (\mathrm{curl}\,\delta E_N,p_B^N)
\Big)\,dt \\
=\int_0^T (\mathrm{curl}\,w,\,p_E^N)\,dt. \tag{$\star\star$}
\end{align*}
From \((\star\star)\) and \((\star)\), 
we obtain the discrete duality identity
\[
\int_0^T\!\!\big[(q_E,\delta E_N)+(q_B,\delta B_N)\big]\,dt
=\int_0^T (\mathrm{curl}\,w,\,p_E^N)\,dt.
\tag{$\ddagger$}
\]
This is the key identity used in the passage to the limit to identify $S_0^*$.

\smallskip
\emph{Step 5 (Pass $N\to\infty$).}
Since $\delta E_N\rightharpoonup\delta E$ and $\delta B_N\rightharpoonup\delta B$ weakly in $L^2(0,T;L^2(\Omega))$
and $q_E,q_B\in L^2(0,T;L^2(\Omega))$, the left-hand side of \((\ddagger)\) converges to
\(\int_0^T[(q_E,\delta E)+(q_B,\delta B)]\,dt\) and right-hand side converges to 
$\int_0^T (\mathrm{curl}\, w, p_E) \, dt$. 
Hence we have:
\[
\int_0^T\!\!\big[(q_E,\delta E)+(q_B,\delta B)\big]\,dt
=\int_0^T (\mathrm{curl}\,w,\,p_E)\,dt.
\]
which is exactly \eqref{eq:S0star_identity}.

\smallskip
\emph{Step 6 (Identification of $S_0^*$).}
By definition of the Hilbert adjoint,
\[
\begin{aligned}
\langle S_0^*(q_E,q_B),\,w\rangle_{Z^*,Z}
&=\int_0^T\!\!\big[(q_E,\delta E)+(q_B,\delta B)\big]\,dt \\
&=\int_0^T (p_E,\mathrm{curl}\,w)\,dt
= \langle \mathcal{C}^* p_E , w \rangle_{Z^*,Z} .
\end{aligned}
\]
Thus $S_0^*(q_E,q_B)=\mathcal{C}^*p_E$ as a functional in $Z^*$.
\end{proof}

\subsection{First–order optimality conditions}
\label{sec:foc}

\begin{theorem}[Optimality system]
\label{thm:optimality_system_new}
Let $\bar z\in Z$ be the unique minimizer from Theorem~\ref{thm:exist-unique-ocp}, and let $(\bar E,\bar B)=S\bar z$.
Then there exists a unique adjoint $(\bar p_E,\bar p_B)$ in the sense of Definition~\ref{def:adjoint_veryweak} with data
$(q_E,q_B)=(\varepsilon(\bar E-E_d),\ \mu^{-1}(\bar B-B_d))$ such that the variational optimality condition
\begin{equation}\label{eq:var_opt_new}
\alpha_1\int_0^T (\bar z,w)\,dt
+\alpha_2\int_0^T (\mathrm{curl}\,\bar z,\mathrm{curl}\,w)\,dt
+\int_0^T (\bar p_E,\mathrm{curl}\,w)\,dt
=0\qquad\forall\,w\in Z
\end{equation}
holds. Equivalently, in $Z^*$,
\begin{equation}\label{eq:Euler_in_Zstar_new}
\alpha_1\,\bar z\;+\;\mathcal{C}^*\!\big(\alpha_2\,\mathcal{C}\bar z+\bar p_E\big)=0 \quad\text{in } Z^* 
\qquad \text{with }\ \mathcal{C}z:=\mathrm{curl}\,z,\ \ \langle \mathcal{C}^*v,w\rangle_{Z^*,Z}= \int_0^T (v,\mathrm{curl}\,w) . 
\end{equation}
\end{theorem}
\begin{proof}
\emph{Step 1: Fréchet derivative of the reduced functional.}
Given $\bar z\in Z$ and a direction $w\in Z$, let $(\delta E,\delta B):=S_0 w$ denote the state increment.
By chain rule and linearity of $S_0$ we obtain
\[
\mathcal{J}'(\bar z)[w]
=\int_0^T\!\big( (\varepsilon(\bar E-E_d),\delta E)+(\mu^{-1}(\bar B-B_d),\delta B)  \big)\,dt
+\alpha_1\int_0^T (\bar z,w)\,dt
+\alpha_2\int_0^T (\mathrm{curl}\,\bar z,\mathrm{curl}\,w)\,dt.
\]

\emph{Step 2: Duality reduction via the adjoint.}
Let $(\bar p_E,\bar p_B)$ be the adjoint associated with $(q_E,q_B)=(\varepsilon(\bar E-E_d),\ \mu^{-1}(\bar B-B_d))$
in the weak sense of \eqref{eq:adjoint_identity}. By Proposition~\ref{prop:S0star} (characterization of $S_0^*$),
for every $w\in Z$,
\[
 \int_0^T\!\big( (\varepsilon(\bar E-E_d),\delta E)+(\mu^{-1}(\bar B-B_d),\delta B)\big)\,dt 
=\int_0^T (\bar p_E,\mathrm{curl}\,w)\,dt .
\]
Therefore
\[
\mathcal{J}'(\bar z)[w]
=\alpha_1\int_0^T (\bar z,w)\,dt
+\alpha_2\int_0^T (\mathrm{curl}\,\bar z,\mathrm{curl}\,w)\,dt
+\int_0^T (\bar p_E,\mathrm{curl}\,w)\,dt .
\]

\emph{Step 3: First–order condition.}
Optimality of $\bar z$ gives $\mathcal{J}'(\bar z)[w]=0$ for all $w\in Z$, which is precisely
\eqref{eq:var_opt_new}. Using the definition of the Hilbert adjoint $\mathcal{C}^* : L^2(0,T;L^2(\Omega)^3)\to Z^*$,
\[
\int_0^T (\bar p_E,\mathrm{curl}\,w)\,dt=\langle \mathcal{C}^*\bar p_E,\,w\rangle_{Z^*,Z},
\qquad
\int_0^T (\mathrm{curl}\,\bar z,\mathrm{curl}\,w)\,dt
=\langle \mathcal{C}^*\mathcal{C}\bar z,\,w\rangle_{Z^*,Z},
\]
hence \eqref{eq:var_opt_new} is equivalent to the operator equation \eqref{eq:Euler_in_Zstar_new} in $Z^*$.

\emph{Sufficiency (and uniqueness).}
Since $\mathcal{J}$ is strictly convex and coercive on the Hilbert space $Z$, 
the first–order condition is also sufficient, so the solution to \eqref{eq:var_opt_new}–\eqref{eq:Euler_in_Zstar_new}
is the unique global minimizer~$\bar z$.
\end{proof}

\section{Fully discrete optimal control problem}
\label{sec:discrete-ocp}

We now develop and analyze a structure-preserving discretization for the solution of the optimal control problem.

\paragraph{Discrete control space and norms.}
We choose the control in the N\'ed\'elec space, piecewise constant in time.
Let
\[
  \mathcal Z_h \;:=\; \mathcal N_h, \qquad
  \mathcal Z_{h,\Delta t}
  \;:=\;\bigl\{\, z_h^{1/2}:\ (0,T]\to \mathcal Z_h \;\text{right–continuous, }
  z_h^{1/2}(t)=z_h^{n+\frac12}\in \mathcal Z_h \text{ on } (t^n,t^{n+1}] \,\bigr\}.
\]
For $z_h^{1/2}\in\mathcal Z_{h,\Delta t}$ we use the midpoint (Bochner) norms
\[
\|z_h^{1/2}\|_{L^2(0,T;L^2(\Omega))}^2 := \sum_{n=0}^{N_t-1}\Delta t\,\|z_h^{n+\frac12}\|_{L^2(\Omega)}^2,
\quad
\|\curl z_h^{1/2}\|_{L^2(0,T;L^2(\Omega))}^2 := \sum_{n=0}^{N_t-1}\Delta t\,\|\curl z_h^{n+\frac12}\|_{L^2(\Omega)}^2.
\]
We write $\|z_h^{1/2}\|_{Z}^2 := \alpha_1 \|z_h^{1/2}\|_{L^2(0,T;L^2(\Omega))}^2 + \alpha_2 \|\curl z_h^{1/2}\|_{L^2(0,T;L^2(\Omega))}^2$.

\paragraph{Discrete state with control.}
Given $z_h^{1/2}\in\mathcal Z_{h,\Delta t}$, we solve, for $n=0,\dots,N_t-1$, the CN step
\begin{subequations}\label{eq:CN-ctrl}
\begin{align}
(\varepsilon\,\delta_t E_h^{n+1},\psi_h)
+(\sigma\, E_h^{n+\frac12},\psi_h)
-(\mu^{-1} B_h^{n+\frac12},\curl\,\psi_h)
&=(I_{\rm src}^{\,n+\frac12}+\curl z_h^{n+\frac12},\psi_h),
\quad \forall\,\psi_h\in \mathcal N_h^0,\\
(\delta_t B_h^{n+1},\phi_h)
+(\curl\, E_h^{n+\frac12},\phi_h)&=0,
\qquad\qquad\qquad\qquad\qquad\ \ \ \forall\,\phi_h\in  \mathcal{RT}_h,
\end{align}
\end{subequations}
with initial data \eqref{eq:initialdiscdata}. Existence/uniqueness and the energy identity of Lemma~\ref{lem:CN-wp} hold verbatim with $f^{\,n+\frac12}=I_{\rm src}^{\,n+\frac12}+\curl z_h^{n+\frac12}$.

\paragraph{Discrete cost functional.}
Let the desired fields be given as midpoint lifts $(E_d^{1/2},B_d^{1/2})$. The discrete
reduced cost reads
\begin{equation}\label{eq:Jhd}
\mathcal J_{h,\Delta t}(z_h^{1/2})
:= \frac12\sum_{n=0}^{N_t-1}\Delta t\,
 \Big(\, \|\sqrt{\varepsilon} (E_h^{n+\frac12}-E_d^{n+\frac12})\|_{L^2(\Omega)}^2
      + \|\sqrt{\mu^{-1}}(B_h^{n+\frac12}-B_d^{n+\frac12})\|_{L^2(\Omega)}^2 \,\Big) 
+ \frac12\,\|z_h^{1/2}\|_{Z}^2,
\end{equation}
where $(E_h^{n+\frac12},B_h^{n+\frac12}) \in \mathcal N_h^0 \times   \mathcal{RT}_h$ is the CN–state \eqref{eq:CN-ctrl} driven by $z_h^{1/2}\in\mathcal Z_{h,\Delta t}$.

\begin{theorem}[Existence and uniqueness of a fully discrete optimal control]\label{thm:disc-exist}
Under Assumption~\ref{ass:data} and $\alpha_1,\alpha_2>0$, the problem
\[
\min_{z_h^{1/2}\in\mathcal Z_{h,\Delta t}} \ \mathcal J_{h,\Delta t}(z_h^{1/2})
\]
admits a unique minimizer $\bar z_h^{1/2}\in\mathcal Z_{h,\Delta t}$.
\end{theorem}

\begin{proof}
$\mathcal J_{h,\Delta t}$ is a strictly convex quadratic functional of $z_h^{1/2}$ in a finite–dimensional space, and the CN state depends \emph{affinely} on $z_h^{1/2}$ via $I_{\rm src}^{\,n+\frac12}+\curl z_h^{n+\frac12}$. Coercivity follows from the $\|z_h^{1/2}\|_{Z}^2$–term. Standard finite–dimensional convex analysis yields existence and uniqueness.
\end{proof}

\subsection{Fully discrete optimality system}

To characterize $\bar z_h^{1/2}$ we introduce the discrete adjoint running \emph{backward} in time. Define backward difference and midpoints
\[
\delta_t^- p^{n}:=\frac{p^{n}-p^{n+1}}{\Delta t},\qquad
p^{n+\frac12}:=\tfrac12(p^{n}+p^{n+1}).
\]

\begin{theorem}[Fully discrete first–order optimality system]\label{thm:disc-opt}
Let $\bar z_h^{1/2}\in\mathcal Z_{h,\Delta t}$ be the unique minimizer of $\mathcal J_{h,\Delta t}$, and let $(\bar E_h^n,\bar B_h^n)$ be the associated CN state \eqref{eq:CN-ctrl}. Then there exist unique adjoint sequences $(\bar p_{E,h}^n,\bar p_{B,h}^n)\in \mathcal N_h^0\times \mathcal{RT}_h$, $n=0,\dots,N_t$, with terminal data
\[
\bar p_{E,h}^{N_t}=0\in \mathcal N_h^0,\qquad \bar p_{B,h}^{N_t}=0\in \mathcal{RT}_h,
\]
solving, backward for $n=N_t-1,\dots,0$,
\begin{subequations}\label{eq:disc-adj}
\begin{align}
(\varepsilon\,\delta_t^- \bar p_{E,h}^{n},\psi_h)
+(\sigma\,\bar p_{E,h}^{n+\frac12},\psi_h)
+(\curl\,\psi_h,\bar p_{B,h}^{n+\frac12})
&=(\varepsilon(\bar E_h^{n+\frac12}-E_d^{n+\frac12}),\psi_h),
&&\forall\,\psi_h\in \mathcal N_h^0,\\
(\delta_t^- \bar p_{B,h}^{n},\phi_h)
-(\mu^{-1}\phi_h,\curl\,\bar p_{E,h}^{n+\frac12})
&=(\mu^{-1}(\bar B_h^{n+\frac12}-B_d^{n+\frac12}),\phi_h),
&&\forall\,\phi_h\in \mathcal{RT}_h,
\end{align}
\end{subequations}
and the discrete stationarity condition
\begin{equation}\label{eq:disc-grad}
\alpha_1\sum_{n=0}^{N_t-1}\Delta t\,(\bar z_h^{n+\frac12},w_h^{n+\frac12})
+\alpha_2\sum_{n=0}^{N_t-1}\Delta t\,(\curl \bar z_h^{n+\frac12},\curl w_h^{n+\frac12})
+\sum_{n=0}^{N_t-1}\Delta t\,(\bar p_{E,h}^{n+\frac12},\curl w_h^{n+\frac12})=0
\end{equation}
for all $w_h^{1/2}\in \mathcal Z_{h,\Delta t}$. Equivalently, time–cellwise,
\begin{equation}\label{eq:disc-pointwise-opt}
\alpha_1\,\bar z_h^{n+\frac12} \;+\; \mathcal{C}_h^*\!\big(\alpha_2\,\curl \bar z_h^{n+\frac12}+P_h\,  \bar p_{E,h}^{n+\frac12}\big)=0
\quad\text{in } \mathcal{N}_h,
\qquad n=0,\dots,N_t-1,
\end{equation}
where $\mathcal{C}_h^*:\mathcal{RT}_h\to \mathcal{N}_h$ is the (discrete) adjoint of curl defined by
\((\mathcal{C}_h^* q_h, w_h)=(q_h,\curl w_h)\) for all $q_h\in \mathcal{RT}_h$, $w_h\in \mathcal N_h^0$. 
\end{theorem}

\begin{proof}
Set up the discrete Lagrangian using the lifted-in-time form \eqref{eq:CN-lift} with forcing
$I_{\rm src}^{1/2}+\curl z_h^{1/2}$ and employ the discrete IBP Lemma~\ref{lem:disc-IBP}. Stationarity
with respect to $E_h^{\mathrm{pl}},B_h^{\mathrm{pl}}$ gives the backward CN adjoint \eqref{eq:disc-adj}, and
stationarity with respect to $z_h^{1/2}$ yields \eqref{eq:disc-grad}–\eqref{eq:disc-pointwise-opt}.
Uniqueness of the adjoint follows as for the forward CN step. 
\end{proof}

\subsection{Convergence of fully discrete optimal controls}

Let $S:Z\to Y$ denote the continuous control-to-state map from Lemma~\ref{lem:S-bounded},
and let $S_{h,\Delta t}:\mathcal Z_{h,\Delta t}\to Y$ denote the map sending $z_h^{1/2}$ to the lifted CN state
$(E_h^{1/2},B_h^{1/2})$ (viewed as elements of $L^2(0,T;L^2(\Omega)^3)$ via the midpoint lift).

\begin{lemma}[Consistency of the discrete control-to-state map]\label{lem:consistency-S}
Suppose $z\in Z$ and let $z_h^{1/2}\in\mathcal Z_{h,\Delta t}$ satisfy
\[
z_h^{1/2}\rightharpoonup z \quad\text{in }L^2(0,T;L^2(\Omega)),\qquad
\curl z_h^{1/2}\rightharpoonup \curl z \quad\text{in }L^2(0,T;L^2(\Omega)),
\]
and $I_{\rm src}^{1/2}\to I_{\rm src}$ in $L^2(0,T;L^2(\Omega))$. Then,
up to a subsequence, the associated CN states satisfy
\[
E_h^{1/2}\to E,\quad B_h^{1/2}\to B \quad\text{in the duality with }L^1(0,T;L^2(\Omega)),
\]
and the lifted piecewise–linear states $E_h^{\mathrm{pl}}\stackrel{*}{\rightharpoonup} E$,
$B_h^{\mathrm{pl}}\stackrel{*}{\rightharpoonup} B$ in $L^\infty(0,T;L^2(\Omega))$, where $(E,B)=S(z)$ is the unique weak solution of \eqref{eq:state-ocp}.
\end{lemma}

\begin{proof}
Apply Theorem~\ref{thm:FD-convergence} with the (varying) right–hand side
$f_h^{1/2}:=I_{\rm src}^{1/2}+\curl z_h^{1/2}$, which is bounded in $L^2(0,T;L^2(\Omega))$ and converges weakly to $f:=I_{\rm src}+\curl z$.
The argument in Lemma~\ref{lem:pl-vs-midpoint} is linear and passes to weakly convergent sources.
Uniqueness of the weak limit identifies the limit with $S(z)$.
\end{proof}

\begin{theorem}
\label{thm:disc-ctrl-conv-operator}
Let Assumption~\ref{ass:data} hold. Let $\bar z\in Z$ be the unique minimizer of the
continuous reduced problem \eqref{eq:reduced} with associated state $(\bar E,\bar B)=S(\bar z)$
and adjoint $(\bar p_E,\bar p_B)$ (Definition~\ref{def:adjoint_veryweak}). Let
$\bar z_{h,\Delta t}^{1/2}\in \mathcal Z_{h,\Delta t}:=L^2(0,T;\mathcal N_h^0)$ be the unique minimizer of
the fully discrete reduced functional $\mathcal J_{h,\Delta t}$ (built with the CN state solver,
midpoint lifts, and discrete tracking). Then, as $h,\Delta t\to 0$ (up to a not–relabeled subsequence),
\[
\bar z_{h,\Delta t}^{1/2}\rightharpoonup \bar z \quad\text{in } L^2(0,T;L^2(\Omega)^3),
\qquad
\curl \bar z_{h,\Delta t}^{1/2}\rightharpoonup \curl \bar z \quad\text{in } L^2(0,T;L^2(\Omega)^3).
\]
Moreover, if $\alpha_1>0$ and $\alpha_2>0$, then the convergence is strong in $Z$:
\[
\bar z_{h,\Delta t}^{1/2}\to \bar z \ \text{ in }L^2(0,T;L^2(\Omega)),
\qquad
\curl \bar z_{h,\Delta t}^{1/2}\to \curl \bar z \ \text{ in }L^2(0,T;L^2(\Omega)).
\]
\end{theorem}

\begin{proof}
\textbf{Step 1 (continuous and discrete optimality equations).}
Let $S_0:Z\to Y$ denote the linear control–to–state map and $S(z)=S_0 z + y^{\rm src}$
with $y^{\rm src}:=(E^{\rm src},B^{\rm src})$. Set $\mathcal C z:=\curl z$ and define the adjoint
$\mathcal C^*$ by $\langle \mathcal C^* v,w\rangle_{Z^*,Z}=\int_0^T (v,\curl w)\,dt$.
For $\mathcal J(z)=\tfrac12\|S(z)-y_d\|_Y^2+\tfrac{\alpha_1}{2}\|z\|_{L^2(\Omega)}^2
+\tfrac{\alpha_2}{2}\|\curl z\|_{L^2(\Omega)}^2$, its derivative is
\[
\mathcal J'(z)[w]=\big(S_0 z + y^{\rm src}-y_d,\, S_0 w\big)_Y
+ \alpha_1 (z,w) + \alpha_2 (\curl z,\curl w).
\]
Hence the continuous optimality condition at $\bar z$ reads
\begin{equation}\label{eq:opt-cont-operator}
\big(S_0 \bar z + y^{\rm src}-y_d,\, S_0 w\big)_Y
+ \alpha_1 (\bar z,w) + \alpha_2 (\curl \bar z,\curl w)=0
\qquad\forall\, w\in Z.
\end{equation}
Equivalently, $H\bar z=r$ in $Z^*$ with
\[
H:=S_0^*\mathcal{W}S_0+\alpha_1 I+\alpha_2\mathcal C^*\mathcal C \quad\text{(bounded, symmetric, coercive)}, 
\qquad r:=S_0^*\mathcal{W}(y_d-y^{\rm src}).
\]
where $\mathcal{W}:= (\varepsilon, \mu^{-1})$ represents the weight tensors' tuple.

Let $S_{0,h,\Delta t}:\mathcal Z_{h,\Delta t}\to Y$ be the discrete linear control–to–state map
(CN scheme, midpoint lift) and $y^{\rm src}_{h,\Delta t}$ the discrete source state.
Then the discrete Euler equation is
\begin{equation}\label{eq:opt-disc-operator}
\big(S_{0,h,\Delta t}\bar z_{h,\Delta t}^{1/2}+y^{\rm src}_{h,\Delta t}-\Pi_{\Delta t}y_d,\,
S_{0,h,\Delta t} w_{h,\Delta t}\big)_Y
+ \alpha_1(\bar z_{h,\Delta t}^{1/2},w_{h,\Delta t})
+ \alpha_2(\curl \bar z_{h,\Delta t}^{1/2},\curl w_{h,\Delta t})=0
\end{equation}
for all $w_{h,\Delta t}\in \mathcal Z_{h,\Delta t}$. This is $H_{h,\Delta t}\bar z_{h,\Delta t}^{1/2}=r_{h,\Delta t}$
in $\mathcal Z_{h,\Delta t}^*$ with
\[
H_{h,\Delta t}:=S_{0,h,\Delta t}^* \mathcal{W} S_{0,h,\Delta t}+\alpha_1 I+\alpha_2\mathcal C^*\mathcal C, 
\qquad
r_{h,\Delta t}:=S_{0,h,\Delta t}^* \mathcal{W} (\Pi_{\Delta t}y_d-y^{\rm src}_{h,\Delta t}),
\]
uniformly bounded and coercive.

\medskip
\textbf{Step 2 (uniform control bounds and weak compactness).}
Testing \eqref{eq:opt-disc-operator} with $w_{h,\Delta t}=\bar z_{h,\Delta t}^{1/2}$ gives
\[
\|S_{0,h,\Delta t}\bar z_{h,\Delta t}^{1/2}\|_Y^2
+ \alpha_1\|\bar z_{h,\Delta t}^{1/2}\|_{L^2(\Omega)}^2
+ \alpha_2\|\curl \bar z_{h,\Delta t}^{1/2}\|_{L^2(\Omega)}^2
= \big(\Pi_{\Delta t}y_d-y^{\rm src}_{h,\Delta t},\,S_{0,h,\Delta t}\bar z_{h,\Delta t}^{1/2}\big)_Y.
\]
By Cauchy--Schwarz/Young and stability of the discrete solver (Lemma~\ref{lem:unifboundEBpl}),
\[
\frac12\|S_{0,h,\Delta t}\bar z_{h,\Delta t}^{1/2}\|_Y^2
+ \alpha_1\|\bar z_{h,\Delta t}^{1/2}\|_{L^2(\Omega)}^2
+ \alpha_2\|\curl \bar z_{h,\Delta t}^{1/2}\|_{L^2(\Omega)}^2
\ \le\ \frac12\|\Pi_{\Delta t}y_d-y^{\rm src}_{h,\Delta t}\|_Y^2
\ \le\ C,
\]
and so there exists $\tilde z\in Z$ and a (not relabeled) subsequence such that
\begin{equation}\label{eq:weak-lims}
\bar z_{h,\Delta t}^{1/2}\rightharpoonup \tilde z \text{ in }L^2(0,T;L^2(\Omega)), 
\qquad
\curl \bar z_{h,\Delta t}^{1/2}\rightharpoonup \curl \tilde z \text{ in }L^2(0,T;L^2(\Omega)).
\end{equation}

\medskip
\textbf{Step 3 (limit in the optimality equation).}
Fix $w\in Z$ with $w\in C_c^\infty(0,T;H^1(\Omega)^3 \cap H(\curl;\Omega))$ (dense in $Z$). 
Let $I_h^N$ be the standard first–order N\'ed\'elec interpolant on $H^1(\Omega)^3\cap H(\curl;\Omega)$, and define
\[
w_{h,\Delta t}(t):=\frac{1}{\Delta t}\int_{I_n} I_h^N w(s)\,ds,\qquad t\in I_n.
\]
Then
\begin{equation}\label{eq:wh-strong}
w_{h,\Delta t}\to w \text{ in }L^2(0,T;L^2(\Omega)),
\qquad
\curl w_{h,\Delta t}\to \curl w \text{ in }L^2(0,T;L^2(\Omega)).
\end{equation}
By the consistency of the discrete forward map (Lemma~\ref{lem:pl-vs-midpoint}, Theorem~\ref{thm:FD-convergence}),
\[
S_{0,h,\Delta t}\bar z_{h,\Delta t}^{1/2}\rightharpoonup S_0\tilde z \text{ in }Y,
\qquad
S_{0,h,\Delta t} w_{h,\Delta t}\to S_0 w \text{ in }Y.
\]
Next, we show the above stated strong convergence. 
Let the discrete forcing be $f_{h,\Delta t}^{1/2}:=\Pi_{\Delta t}I_{\rm src}+\curl w_{h,\Delta t}$ and the
continuous one $f:=I_{\rm src}+\curl w$. Then
$f_{h,\Delta t}^{1/2}\to f$ strongly in $L^2(0,T;L^2(\Omega)^3)$.
Denote the corresponding discrete state (midpoint lift) by
$y_{h,\Delta t}:=S_{0,h,\Delta t} w_{h,\Delta t}=(E_{h,\Delta t}^{1/2},B_{h,\Delta t}^{1/2})$
and the continuous state by $y:=S_0 w=(E,B)$.

By Lemma~\ref{lem:CN-wp} (summed in time) we have the discrete energy balance
\begin{equation}\label{eq:disc-bal}
\|y_{h,\Delta t}\|_Y^2 \;+\; \int_0^T \|\sqrt{\sigma}\,E_{h,\Delta t}^{1/2}\|_{L^2(\Omega)}^2\,dt
\;=\; \|y_{h,\Delta t}(0)\|_Y^2 \;+\; \int_0^T (f_{h,\Delta t}^{1/2},\,E_{h,\Delta t}^{1/2})\,dt,
\end{equation}
while the continuous solution $y=(E,B)=S_0 w$ satisfies
\begin{equation}\label{eq:cont-bal}
\|y\|_Y^2 \;+\; \int_0^T \|\sqrt{\sigma}\,E\|_{L^2(\Omega)}^2\,dt
\;=\; \|y(0)\|_Y^2 \;+\; \int_0^T (f,\,E)\,dt .
\end{equation}
Since $y_{h,\Delta t}(0)\to y(0)$ in $Y$, $f_{h,\Delta t}^{1/2}\to f$ in $L^2(0,T;L^2(\Omega)^3)$, and
$E_{h,\Delta t}^{1/2}\rightharpoonup E$ weakly in $L^2(0,T;L^2(\Omega)^3)$. Then
\[
\int_0^T (f_{h,\Delta t}^{1/2},\,E_{h,\Delta t}^{1/2})\,dt \;\longrightarrow\; \int_0^T (f,\,E)\,dt
\quad\text{(strong--weak convergence),}
\]
and by weak lower semicontinuity,
\[
\liminf_{h,\Delta t}\int_0^T \|\sqrt{\sigma}\,E_{h,\Delta t}^{1/2}\|_{L^2(\Omega)}^2\,dt
\;\ge\; \int_0^T \|\sqrt{\sigma}\,E\|_{L^2(\Omega)}^2\,dt .
\]
Taking $\limsup$ in \eqref{eq:disc-bal} and using the previous two limits yields
\[
\limsup_{h,\Delta t}\|y_{h,\Delta t}\|_Y^2
\;\le\; \|y(0)\|_Y^2 + \int_0^T (f,\,E)\,dt \;-\; \liminf_{h,\Delta t}\!\int_0^T \|\sqrt{\sigma}\,E_{h,\Delta t}^{1/2}\|_{L^2(\Omega)}^2\,dt
\;\le\; \|y\|_Y^2,
\]
where the last inequality uses \eqref{eq:cont-bal}. This proves
\[
\limsup_{h,\Delta t}\|y_{h,\Delta t}\|_Y^2 \;\le\; \|y\|_Y^2.
\]
Together with weak lower semicontinuity,
$\|y\|_Y^2 \le \liminf_{h,\Delta t}\|y_{h,\Delta t}\|_Y^2$, we conclude
$\|y_{h,\Delta t}\|_Y \to \|y\|_Y$, and hence
\[
y_{h,\Delta t}\ \to\ y \quad\text{strongly in }Y.
\]
Therefore \(S_{0,h,\Delta t} w_{h,\Delta t}\to S_0 w\) in \(Y\).

Moreover, $y^{\rm src}_{h,\Delta t}\to y^{\rm src}$ and $\Pi_{\Delta t}y_d\to y_d$ in $Y$.
Passing to the limit in \eqref{eq:opt-disc-operator} with the test $w_{h,\Delta t}$ yields
\[
\big(S_0\tilde z+y^{\rm src}-y_d,\,S_0 w\big)_Y
+\alpha_1(\tilde z,w)+\alpha_2(\curl\tilde z,\curl w)=0
\qquad\forall\, w\in C_c^\infty(0,T;H^1(\Omega)^3\cap H(\curl;\Omega)).
\]
By density, the identity holds for all $w\in Z$, and by uniqueness of the solution
to \eqref{eq:opt-cont-operator}, we conclude $\tilde z=\bar z$.

\medskip
\textbf{Step 4 (strong convergence).}
Set
\[
u_{h,\Delta t}:=S_{0,h,\Delta t}\,\bar z_{h,\Delta t}^{1/2}\in Y,
\qquad
a_{h,\Delta t}:=\Pi_{\Delta t}y_d - y^{\rm src}_{h,\Delta t}\in Y .
\]
Testing the discrete optimality equation \eqref{eq:opt-disc-operator} with
$w_{h,\Delta t}=\bar z_{h,\Delta t}^{1/2}$ yields
\begin{equation}\label{eq:key-id-disc}
\|u_{h,\Delta t}\|_Y^2
+ \alpha_1\|\bar z_{h,\Delta t}^{1/2}\|_{L^2(\Omega)}^2
+ \alpha_2\|\curl \bar z_{h,\Delta t}^{1/2}\|_{L^2(\Omega)}^2
\;=\; (a_{h,\Delta t},\,u_{h,\Delta t})_Y .
\end{equation}
Equivalently,
\begin{equation}\label{eq:key-id}
\alpha_1\|\bar z_{h,\Delta t}^{1/2}\|_{L^2(\Omega)}^2
+ \alpha_2\|\curl \bar z_{h,\Delta t}^{1/2}\|_{L^2(\Omega)}^2
\;=\; (a_{h,\Delta t},\,u_{h,\Delta t})_Y - \|u_{h,\Delta t}\|_Y^2 .
\end{equation}

By the compactness/consistency already established,
\[
\bar z_{h,\Delta t}^{1/2}\rightharpoonup \bar z \text{ in }L^2,\qquad
u_{h,\Delta t}=S_{0,h,\Delta t}\bar z_{h,\Delta t}^{1/2}\rightharpoonup S_0\bar z \text{ in }Y,
\qquad
a_{h,\Delta t}\to y_d-y^{\rm src}\ \text{ in }Y.
\]
Hence
\[
(a_{h,\Delta t},u_{h,\Delta t})_Y \longrightarrow (y_d-y^{\rm src},\,S_0\bar z)_Y,
\qquad
\liminf_{h,\Delta t}\|u_{h,\Delta t}\|_Y^2 \;\ge\; \|S_0\bar z\|_Y^2 .
\]
Taking $\limsup$ in \eqref{eq:key-id} and using the elementary inequality
$\limsup(x_n-y_n)\le \limsup x_n - \liminf y_n$ gives
\begin{equation}\label{eq:limsup-pre}
\limsup_{h,\Delta t}\Big[\alpha_1\|\bar z_{h,\Delta t}^{1/2}\|_{L^2(\Omega)}^2
+ \alpha_2\|\curl \bar z_{h,\Delta t}^{1/2}\|_{L^2(\Omega)}^2\Big]
\;\le\; (y_d-y^{\rm src},\,S_0\bar z)_Y - \|S_0\bar z\|_Y^2 .
\end{equation}
Evaluate the continuous optimality equation \eqref{eq:opt-cont-operator} at $w=\bar z$:
\[
(S_0\bar z+y^{\rm src}-y_d,\,S_0\bar z)_Y
+\alpha_1\|\bar z\|_{L^2(\Omega)}^2+\alpha_2\|\curl\bar z\|_{L^2(\Omega)}^2=0,
\]
i.e.,
\[
(y_d-y^{\rm src},\,S_0\bar z)_Y
=\|S_0\bar z\|_Y^2+\alpha_1\|\bar z\|_{L^2(\Omega)}^2+\alpha_2\|\curl\bar z\|_{L^2(\Omega)}^2.
\]
Insert this into \eqref{eq:limsup-pre} to obtain
\[
\limsup_{h,\Delta t\to0}\ \Big[\,\alpha_1\|\bar z_{h,\Delta t}^{1/2}\|_{L^2(\Omega)}^2
+\alpha_2\|\curl \bar z_{h,\Delta t}^{1/2}\|_{L^2(\Omega)}^2\Big]
\ \le\ \alpha_1\|\bar z\|_{L^2(\Omega)}^2+\alpha_2\|\curl \bar z\|_{L^2(\Omega)}^2,
\]
which is the desired estimate.
Together with \eqref{eq:weak-lims} and lower semicontinuity, this implies equality of the norms,
hence strong convergence in the corresponding components. If $\alpha_1>0$ and $\alpha_2 >0$ then we immediately obtain the strong convergence.
\end{proof}

\section{Numerical examples}
\label{sec:numerics}

\label{section:numerics}
We present three numerical examples to demonstrate our structure-preserving optimal control scheme.
Our approach is to solve equation \eqref{eq:CN} forward in time for the states $(E_h^n, B_h^n)$, then solve equation \eqref{eq:disc-adj} backward in time for the adjoint states $(\bar p_{E,h}^n,\bar p_{B,h}^n)$, for all $n$, and compute the discrete gradient in \eqref{eq:disc-grad}. 
We use a quasi-Newton method based on the limited-memory BFGS secant approximation to update the control and iterate until the gradient norm is less than a set value or a maximum number of iterations is exceeded.
We perform our numerical experiments using two high-performance computing frameworks: \emph{MrHyDE: A framework for solving Multi-resolution Hybridized Differential Equations} \cite{mrhyde_user}, to solve Maxwell's equations and compute adjoint sensitivities and gradients, and the \emph{Rapid Optimization Library (ROL)} \cite{ROL2022ICCOPT}, to solve the optimization problem.

\subsection{Example 1}
		We consider as our first example a 1D problem modeled with a pseudo-3D domain with periodic boundary condition surfaces parallel to the $xz$ and $yz$ planes, respectively.
        The domain ($z$-dimension) is given by: 
    $$ \Omega = (-40\times 10^{-6} \,\text{m}\;,\; 40\times 10^{-6} \,\text{m})$$ 
    and subdivided into source, control, and observation regions, as shown in Figure \ref{fig:1D-domain}. This is similar to the experiment from Figure \ref{fig:split_vs_nosplit}. We partition $\Omega$ with 396 equally spaced intervals.
	\begin{figure}[h!]
	    \centering
        \begin{tikzpicture}
            \coordinate (A) at (-6,0);
            \coordinate (B) at (6,0);
            \coordinate (C1) at (-1.8,0);
            \coordinate (C2) at (-0.6,0);
            \coordinate (D1) at (0.6,0);
            \coordinate (D2) at (1.8,0);
            
            \draw[very thick] (A) -- (C1) 
                node[midway, above, sloped, font=\small] {observation};
            \draw[blue, very thick] (C1) -- (C2) 
                node[midway, below, sloped, font=\small] {control};
            \draw[blue, very thick] (D1) -- (D2) 
                node[midway, below, sloped, font=\small] {control};
            \draw[red, very thick] (C2) -- (D1) 
                node[midway, above, sloped, font=\small] {source};
            \draw[very thick] (D2) -- (B) 
                node[midway, above, sloped, font=\small] {observation};
                
            \fill (A) circle (2pt) node[below] {$-40\cdot 10^{-6}$};
            \fill (B) circle (2pt) node[below] {$40\cdot 10^{-6}$};
            \fill[blue] (C1) circle (2pt);
            \fill[red] (C2) circle (2pt);
            \fill[red] (D1) circle (2pt);
            \fill[blue] (D2) circle (2pt);
        \end{tikzpicture}
        \caption{1D setup, where the source numerically models an infinite current sheet in the $xy$ plane, in the middle of the `source' region.  }
	    \label{fig:1D-domain}
	\end{figure}
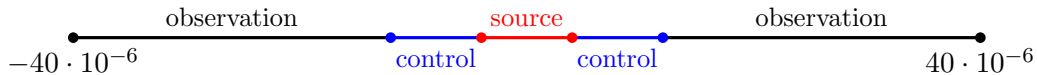

    We take the time interval $(0,T) = (0, 200\cdot 10^{-15} \,\text{s})$ and use 400 time steps ($\Delta t = 0.5\cdot 10^{-15} \,\text{s}$).
    We assume the electromagnetic waves travel in a vacuum.  Throughout, the electric permittivity and magnetic permeability constants are given by $$\varepsilon =8.854187817\cdot 10^{-12} \; \mathrm{farads}/\text{m}  \qquad \text{ and } \qquad \mu = 1.2566370614\cdot 10^{-6} \; \mathrm{henries}/\text{m}, $$ respectively.  	
    We set the electric conductivity $\sigma$ to zero throughout the domain except close to the boundary where we gradually ramp it up so it acts as an adiabatic absorber to prevent reflection of waves back into the domain.
    Given $\Omega_{\rm abs}$, an adiabatic absorption region near the boundary, the formula for the ramp-up function is
	\begin{align}\label{conductivity_def}
	\sigma(z,t) = \left\lbrace\begin{array}{ccl}
	0 & ,& (z,t) \in\overline{\Omega \setminus \Omega_{\rm abs}} \times [0,T] \\
	\sigma_{\max}\left(\frac{|z| - (40\cdot 10^{-6} \,\text{m} - L_{\rm abs})}{L_{\rm abs}}\right)^3 &, & (z,t)\in  \overline{\Omega_{\rm abs}} \times [0,T]
	\end{array}\right.
	\end{align} 
	where $$\sigma_{\max} = 10^{4} \; \mathrm{siemens}/\text{m} \qquad , \qquad L_{\rm abs} = 8 \cdot 10^{-6}\,\text{m}.$$
    For the initial conditions we choose $E_0 = 0$ and $B_0 = 0$.
    The target fields are set to $E_d = 0$ and~$B_d = 0$, i.e., we wish to cloak the electromagnetic source.
    For the source current density, we use 
	$$ I_{\rm src}(z,t) =  (\tilde{I}_{\rm src}(z,t), 0,  0 ) $$ where the source $\tilde{I}_{\rm src}(z,t)$ is supported a small strip in the middle subdomain, 
	\begin{align*}
	\Omega_J = [-0.404\cdot 10^{-6} \,\text{m} \;, \;0.404\cdot 10^{-6} \,\text{m}],
	\end{align*}
    containing four mesh elements, and follows a sinusoidal waveform with Gaussian ramp-up given by (in amperes/$\text{m}^2$)
	\begin{align}\label{source_current_def}
	\tilde{I}_{\rm src}(z,t) = \left\lbrace\begin{array}{lcl}
	0 \enspace  & , & (z,t) \in  \Omega \setminus \Omega_J \times [0,T],  \\
	g(t) \cdot e^{-2[\pi \cdot \sigma_J \cdot (t-t_{\text{offset}})]^2}  & , & (z,t) \in \Omega_J \times [0, t_{\text{offset}}],   \\
	g(t)  & , & (z,t) \in  \Omega_J \times [t_{\text{offset}}, T],
	\end{array}\right.
	\end{align}
	where 
	\begin{align*}
	g(t) &= \cos(2 \pi  \cdot f_{\text{center}} \cdot (t-t_{\text{offset}})), \\
	\sigma_J &= \frac{40}{2\sqrt{2 \log(2)}} \cdot 10^{12}  \,\text{Hz}, \\
	t_{\text{offset}} &= 50 \cdot 10^{-15} \,\text{s}, \\
	f_{\text{center}} &= 75 \cdot 10^{12} \,\text{Hz}.
	\end{align*}

    \begin{figure}[htb!]
        \centering
        \includegraphics[width=0.95\linewidth]{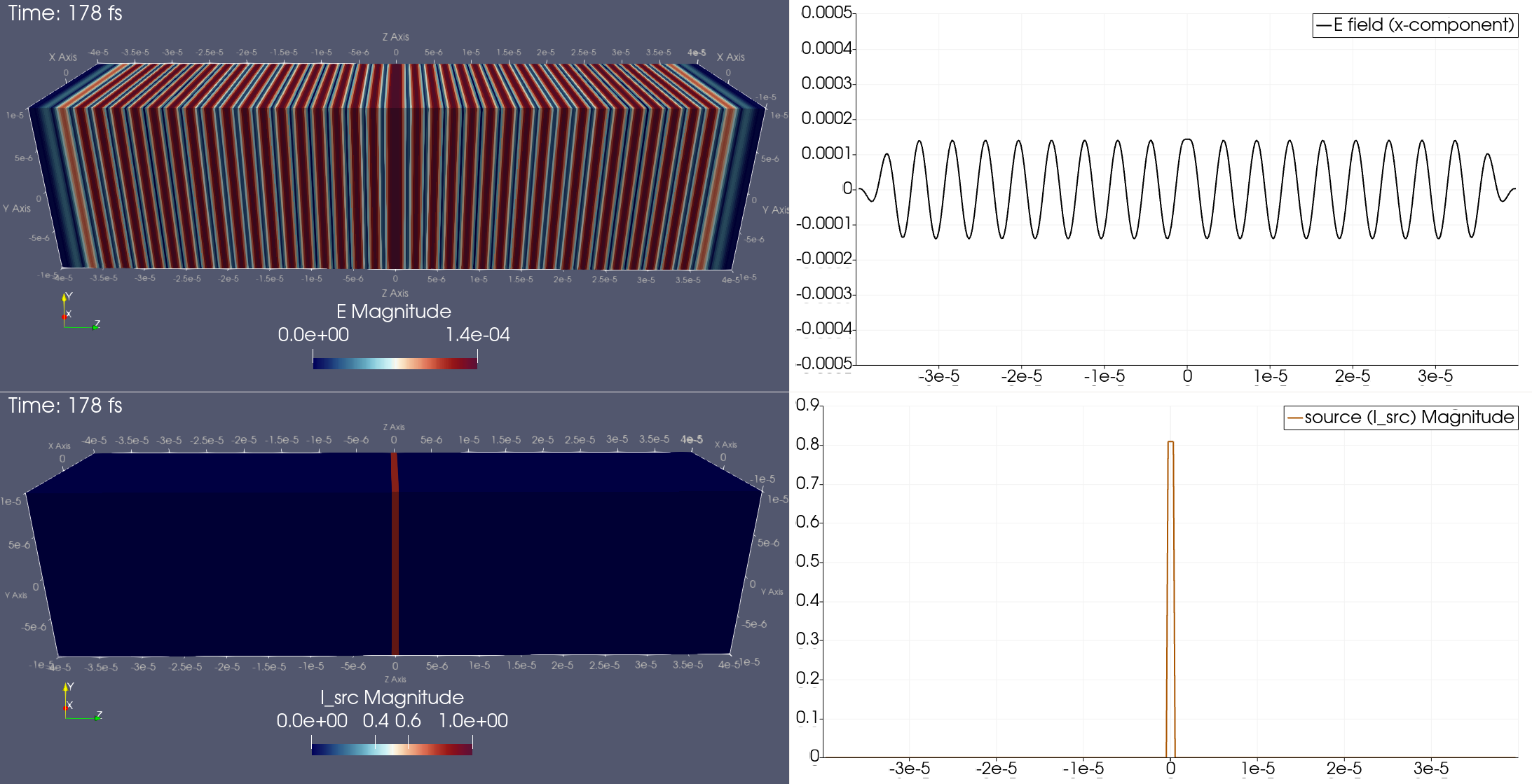}
        \caption{Forward problem: electric field and source current magnitudes at $t=178\,\text{fs}$. The left panel shows the 3D periodic block, with electric field propagation on top and the source current below. The right panel shows the 1D versions obtained by plotting over a line along the $z$-axis. The electric field propagates throughout to entire domain.}
        \label{fig:1D-forward}
        \bigskip
    \end{figure}
    \begin{figure}[htb!]
        \includegraphics[width=0.95\linewidth]{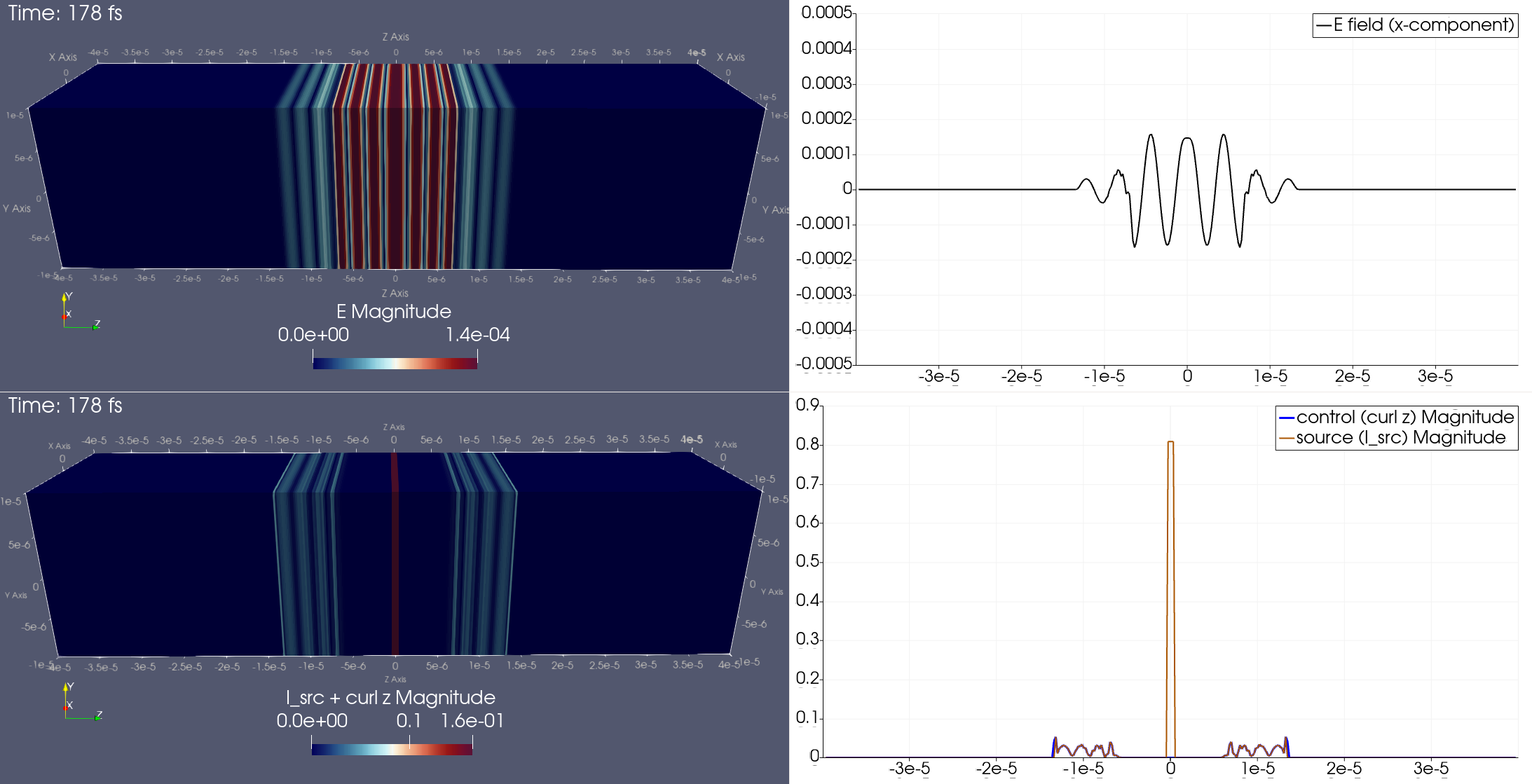}
        \caption{Control problem: electric field, control current and source current magnitudes at $t=178 \,\text{fs}$. The left panel shows the periodic 3D block with electric field propagation on top and $I_\text{src}$ and $\curl z$ magnitudes below. The right panel shows the 1D plots over a line. Here, the bottom image in the right panel shows the source current magnitude in comparison with the control current (curl z) magnitude. The top image shows that the electric field is near zero in observation regions.}
        \label{fig:1D-opt}
    \end{figure}

    In the objective function, we use a tracking weight of $10^{35}$ and control penalties $\alpha_1, \alpha_2 = 10^{5}$.
    The results of the forward problem simulation are shown in Figure~\ref{fig:1D-forward}. Figure~\ref{fig:1D-opt} shows the solution of the optimal control problem. In each figure we show the magnitude of the electric field at a point in time (178 femtoseconds), and the magnitude of the source current (and control current when solving the optimal control problem). We also provide line plots through the z-axis, showing linear interpolants of the cell-centered data provided by MrHyDE.
    The top panel of Figure~\ref{fig:1D-opt} indicates that the electric field generated by the source is blocked very effectively. 


\subsection{Example 2}
    As our second example, we consider a $2$D square domain modeled with a pseudo-3D domain with periodic boundary condition surfaces parallel to the $xy$ plane.
    The domain is given by
    $$ \Omega = (-36\cdot 10^{-6} \,\text{m}\;,\; 36\cdot 10^{-6} \,\text{m})^2 $$ 
    and has four distinct regions, see Figure~\ref{fig:2D-domain}.
    Near the center is a thunderbird source region ($\Omega_J$).
    Next, there are four small disjoint circular control regions ($\Omega_\text{ctrl}$), placed around the source region, where the control current will be generated.
    Surrounding the source and control regions is a circular air region within which we let the source and control electromagnetic fields propagate and interact.
    Finally, surrounding the air region is the observation region ($\Omega_\text{obs}$) where we expect the fields to be close to the target values.
    \begin{figure}[htb!]
	    \centering
             \includegraphics[width=0.8\linewidth]{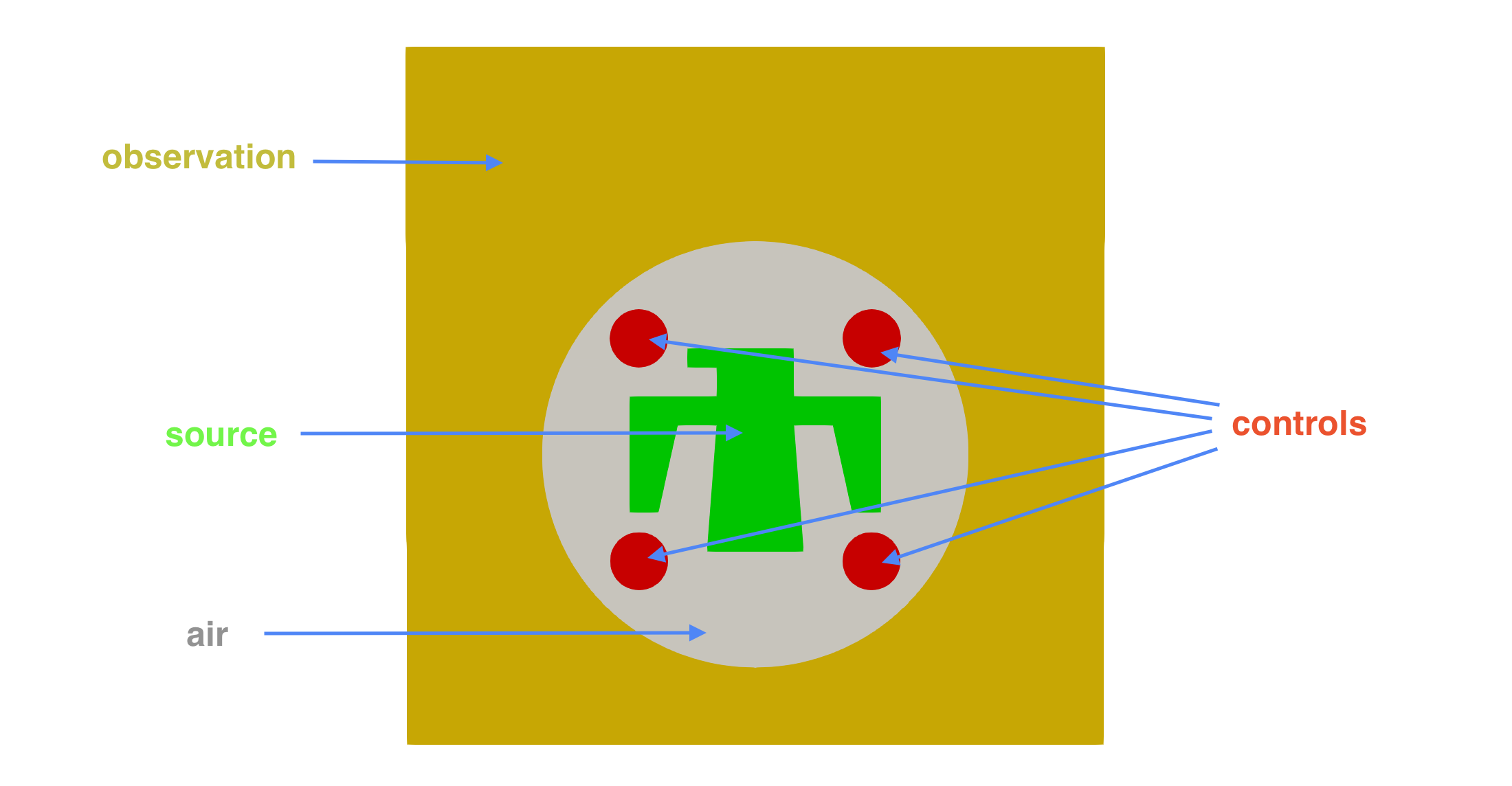}
	    \caption{2D setup, with a `thunderbird' source, four cylindrical controls, and an observation region.}
	    \label{fig:2D-domain}
	\end{figure}

    We discretize the pseudo-3D domain with 47,746 hexahedral elements and solve \eqref{eq:reduced} on a time interval of $(0,T) = (0, 200\cdot 10^{-15} \,\text{s})$ with 400 time steps.
    Electric conductivity is defined as 
    \begin{align}\label{conductivity_def2d}
	\sigma(x,y,t) = \left\lbrace\begin{array}{ccl}
	0 & ,& (x,y,t) \in\overline{\Omega \setminus \Omega_{\rm abs}} \times [0,T] \\
	\sigma_{\max}\left(\frac{s(x)}{L_{\rm abs}}\right)^3 + \sigma_{\max}\left(\frac{s(y)}{L_{\rm abs}}\right)^3  &, & (x,y,t)\in  \overline{\Omega_{\rm abs}} \times [0,T]
	\end{array}\right.
	\end{align} 
    where $s(x_i) = |x_i| - (36\cdot 10^{-6}\,\text{m} - L_{\rm abs}) $ with the adiabatic absorption region length $ L_{\rm abs} = 10 \cdot 10^{-6}\,\text{m} $.  \\
    The initial conditions and the desired fields are again set to zero and the source current is given by 
    $$ I_{\rm src}(x,y,t) =  (0, 0, \tilde{I}_{\rm src}(x,y,t) ) $$ with $\tilde{I}_{\rm src}(x,y,t)$ defined as in \eqref{source_current_def}, using $(x,y)$ instead of $(z)$, and $\Omega_J$ given by the thunderbird source region. The parameters are: 	
    \begin{align*}
	g(t) &= \cos(2 \pi  \cdot f_{\text{center}} \cdot (t-t_{\text{offset}})), \\
	\sigma_J &= \frac{25}{2} \cdot 10^{12}  \,\text{Hz}, \\
	t_{\text{offset}} &= 50 \cdot 10^{-15} \,\text{s}, \\
	f_{\text{center}} &= 42.9 \cdot 10^{12} \,\text{Hz}.
	\end{align*}
    
    We use a tracking weight of $ 10^{35}$ and control penalties $\alpha_1, \alpha_2 = 10^{5}$. Figure~\ref{fig:2D-max} shows the electromagnetic energies in the forward and optimal control simulations as well as the control current. Figure~\ref{fig:2D-lineplot} shows a plot along a line. We observe that the electromagnetic fields are near zero in the observation region.
    \begin{figure}[htb!]
        \centering
        \includegraphics[width=0.95\linewidth]{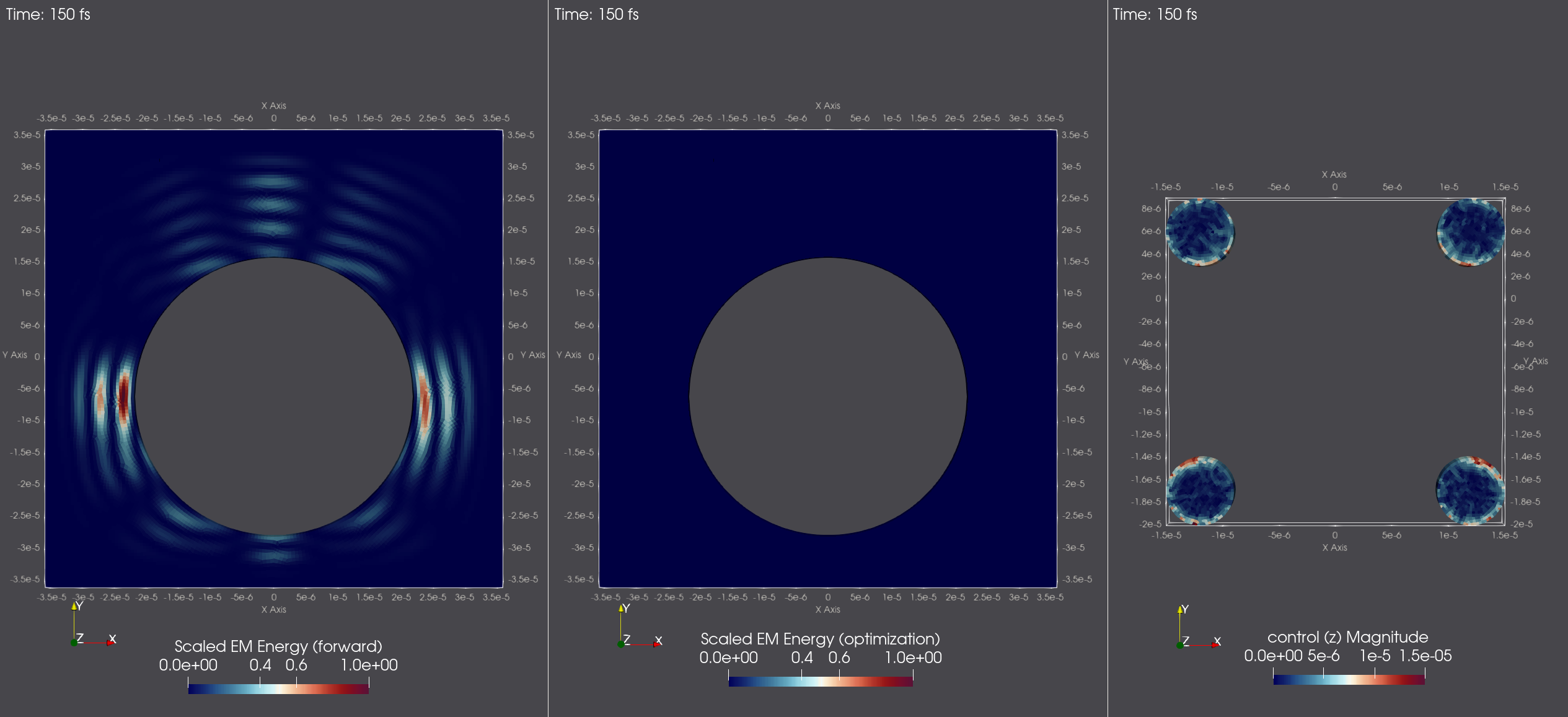}
        \caption{Scaled electromagnetic energies of forward and optimal control simulations in the observation regions at $t=150 \,\text{fs}$. The scaling is done by dividing by the max value. The left image shows the uncontrolled fields, the center image shows the controlled fields, and the right image shows the control current magnitude. The controlled field is near zero.}
        \label{fig:2D-max}
  \end{figure}
  \begin{figure}[htb!]
        \centering
        \includegraphics[width=0.5\linewidth]{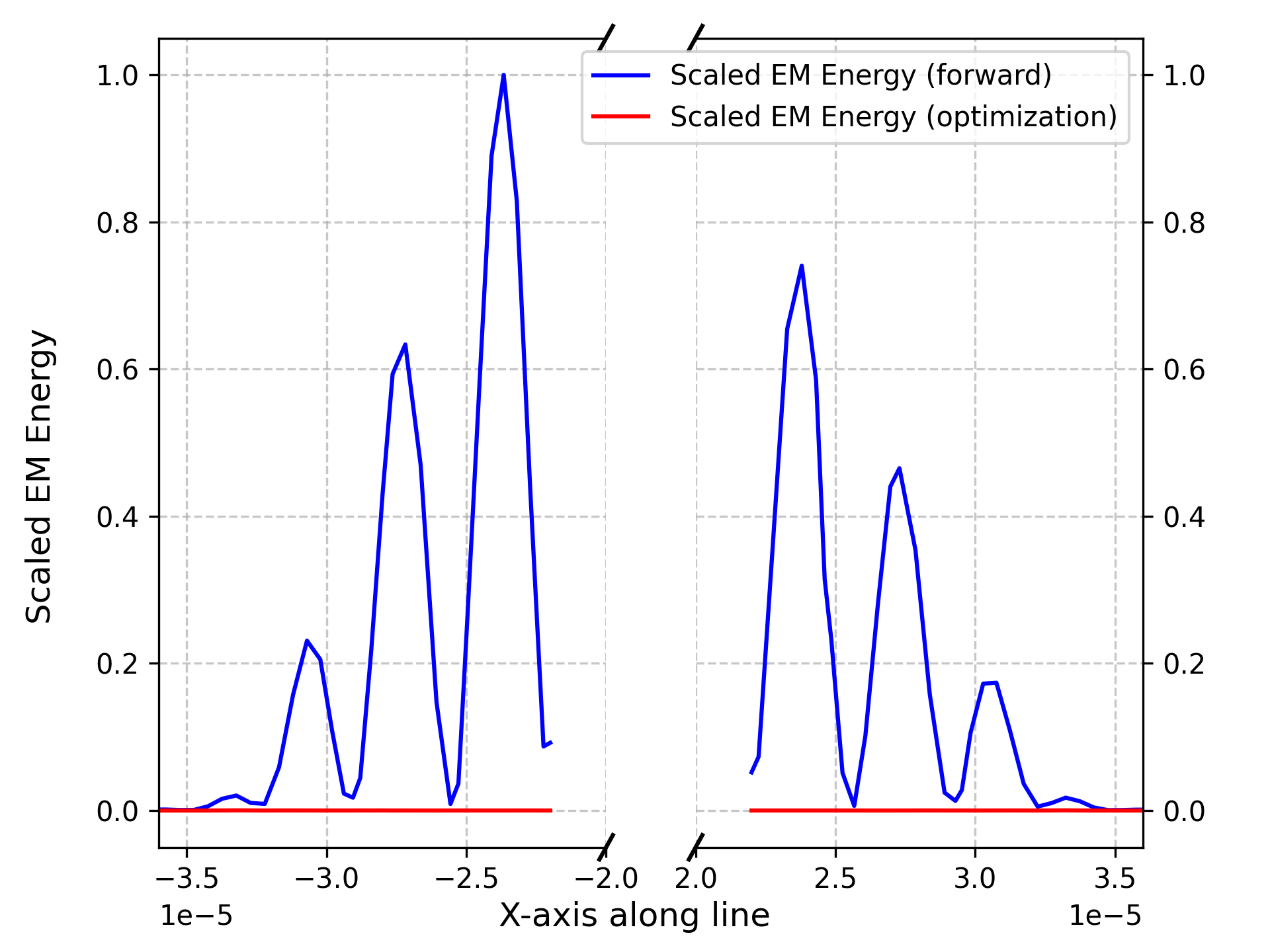}        
        \caption{Line plot of the scaled electromagnetic energies for the forward and optimization simulations at $t=150 \,\text{fs}$ through a line across the $x$-endpoints at $y=-6\cdot 10^{-6} \,\text{m}$. The sharp edges are artifacts of the sampling used for the plot. The middle portion is removed to emphasize the observation region.  The controlled electromagnetic energy (red) is near zero.}
        \label{fig:2D-lineplot}
    \end{figure}

    \subsection{Example 3}
	Our third and final example is a cube domain with dimensions 
    $$ \Omega = [-20\cdot 10^{-6} \,\text{m} \;,\; 20\cdot 10^{-6} \,\text{m}]^3$$ 
    and four distinct regions, see Figure \ref{fig:3D-domain}:
    a `cup' source region, $\Omega_J$, at the center;
    eight small disjoint spherical control regions, $\Omega_{\rm ctrl}$, placed around the source region;
    a spherical air region surrounding the source and control regions within which the source and control fields interact; and
    an observation region, $\Omega_{\rm obs}$, surrounding the air regions, where we expect the controlled fields to match their targets.
	\begin{figure}[htb!]
      \centering
      \includegraphics[width=0.95\linewidth]{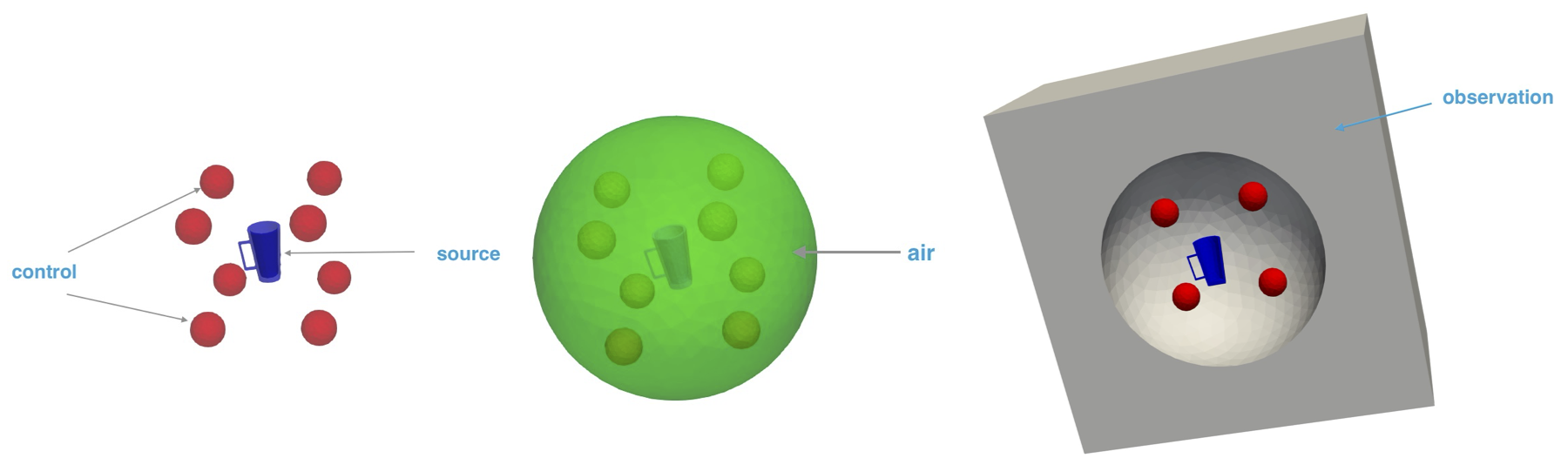}
      \caption{3D setup, with a `cup' source, eight spherical controls, and an observation region.}
      \label{fig:3D-domain}
	\end{figure}    

    We discretize the entire domain with 2,332,860 tetrahedral elements and solve \eqref{eq:reduced} on $(0,T) = (0, 200\cdot 10^{-15} \,\text{s})$ with 1600 time steps.
    The electric conductivity is defined as 
    \begin{align}\label{conductivity_def3d}
	\sigma(x,y,z,t) = \left\lbrace\begin{array}{ccl}
	0 & ,& (x,y,z,t) \in\overline{\Omega \setminus \Omega_{\rm abs}} \times [0,T] \\
	\sigma_{\max}\left[\left(\frac{s(x)}{L_{\rm abs}}\right)^3 + \left(\frac{s(y)}{L_{\rm abs}}\right)^3 + 
    \left(\frac{s(z)}{L_{\rm abs}}\right)^3 \right] &, & (x,y,z,t)\in  \overline{\Omega_{\rm abs}} \times [0,T]
	\end{array}\right.
	\end{align} 
     where $s(x_i) = |x_i| - (20\cdot 10^{-6} \,\text{m} - L_{\rm abs}) $, and the adiabatic absorption region length $ L_{\rm abs} = 10 \cdot 10^{-6} \,\text{m} $.
    The initial conditions and the desired fields are set to zero, and the source current is given by
    $$ I_{\rm src}(x,y,z,t) =  (\tilde{I}_{\rm src}(x,y,z,t), 0, 0 ) $$ with $\tilde{I}_{\rm src}(x,y,z,t$) defined similarly to \eqref{source_current_def}, with $\Omega_J$ given by the cup source region and $g(t), \sigma_J, t_{\rm offset}$ and $f_{\rm center}$ defined as in Example 2.
    
    We use a tracking weight of $ 10^{37}$ and regularization weights $\alpha_1, \alpha_2 = 10^{5}$. To visualize the results we take a diagonal slice through the 3D domain.
    Figure~\ref{fig:3D-max} shows the electromagnetic energies in the forward and optimal control simulations as well as the control current.
    Figure~\ref{fig:3D-lineplot} shows the energy plot along a line through the diagonal slice.
    The source is cloaked very effectively.
    \begin{figure}[htb!]
        \centering
        \includegraphics[width=\linewidth]{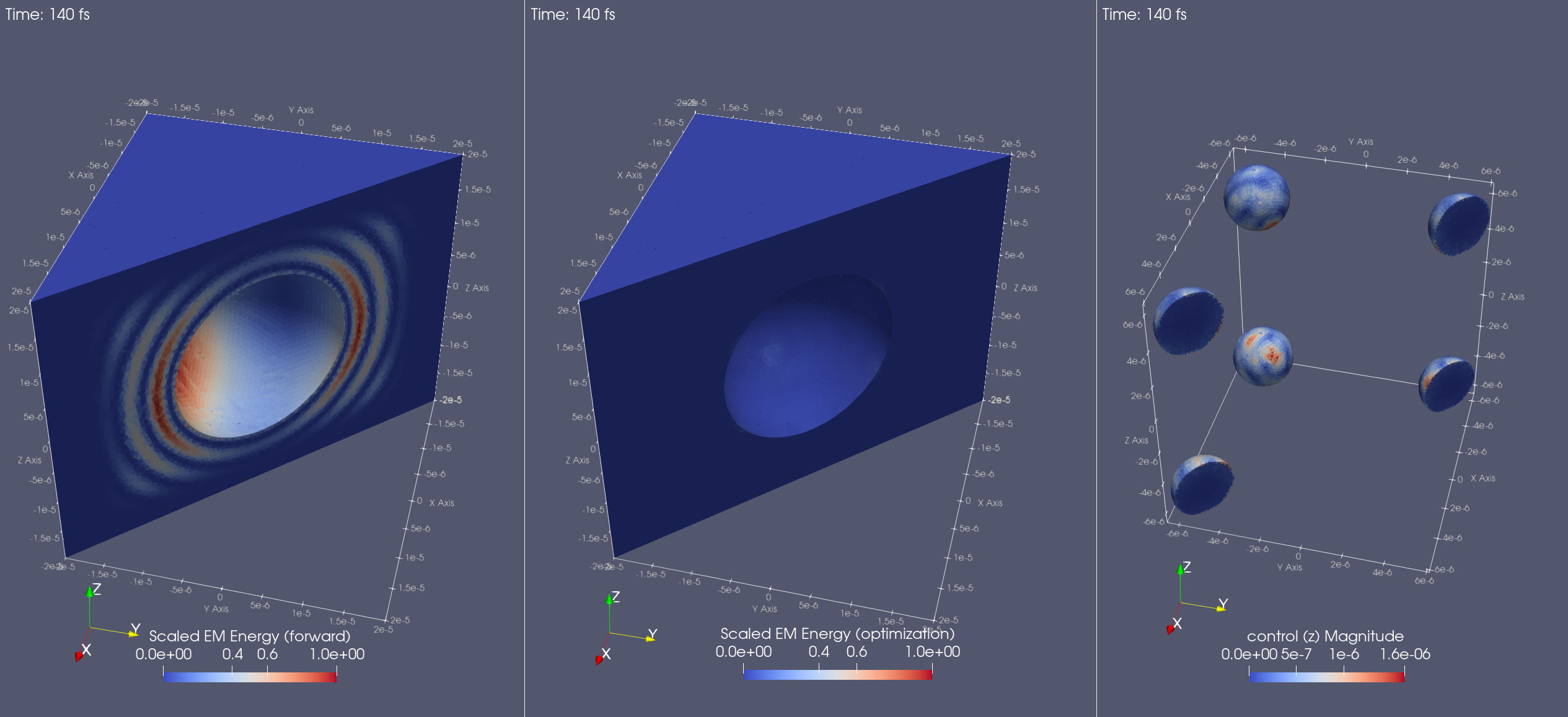}
        \caption{Scaled electromagnetic energies of forward and optimal control simulations in the observation regions at $t=140 \,\text{fs}$. The scaling is done by dividing by the max value. The left image shows the uncontrolled energy, the center image shows the controlled energy, and the right image shows the control current magnitude.  Controlled fields are near zero.}
        \label{fig:3D-max}
    \end{figure}
    \begin{figure}[htb!]
        \centering
        \includegraphics[width=0.5\linewidth]{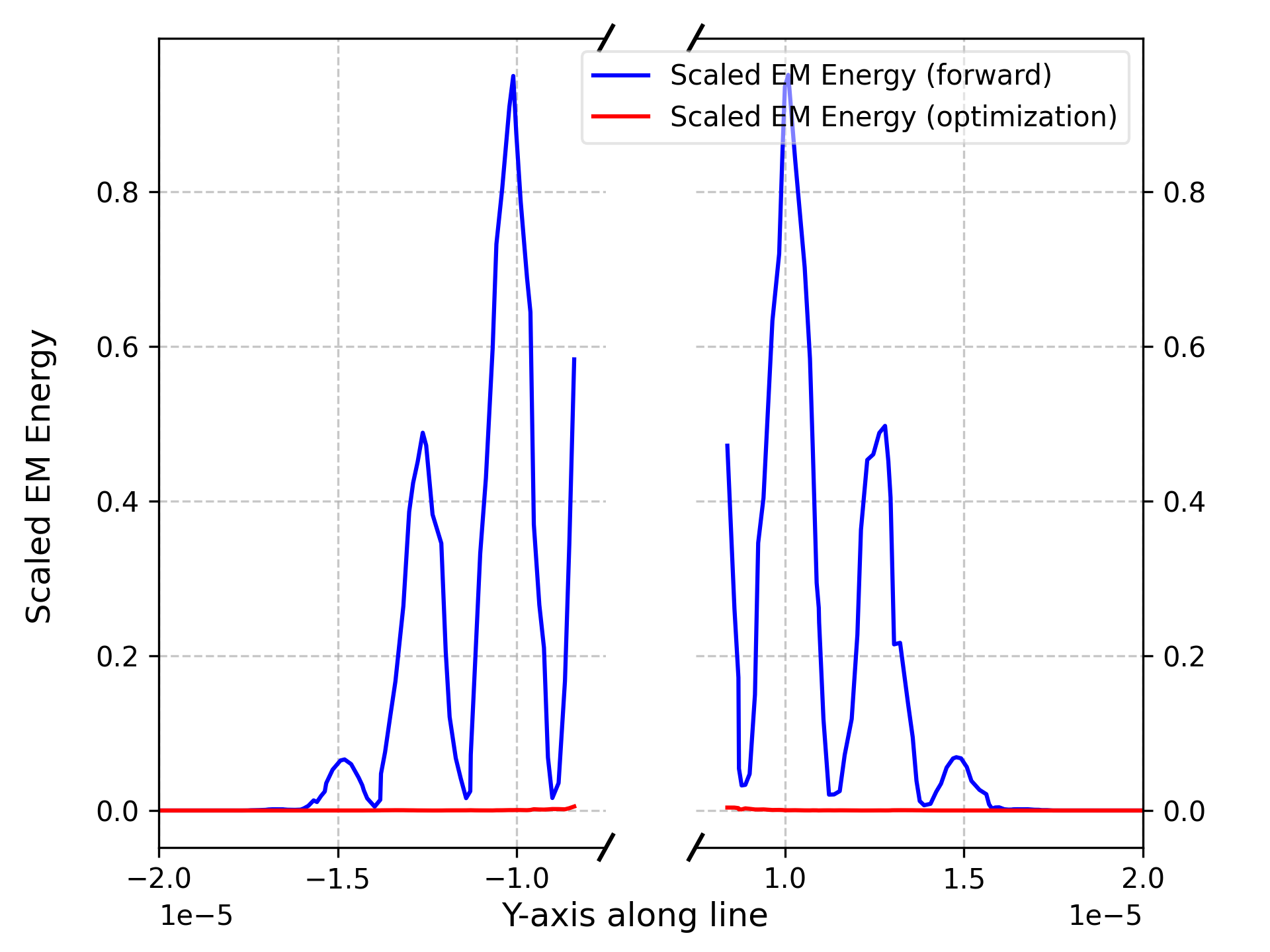}
        \caption{Line plot of the scaled electromagnetic energies for the forward and optimization simulations at $t=140\,\text{fs}$ through a line across the $xy$-plane at $z=-2.5\cdot 10^{-6}\,\text{m}$. The sharp edges are artifacts of the sampling used for the plot. The middle portion is removed to emphasize the observation region.  The controlled electromagnetic energy (red) is near zero.}
        \label{fig:3D-lineplot}
    \end{figure}

\section*{Acknowledgments}
We thank Eric Cyr for a careful reading of the paper and the suggestion to use $\curl z$ instead of $z$ as the control, to conserve charge.

\appendix
\counterwithin{theorem}{section}
\counterwithin{proposition}{section}
\counterwithin{lemma}{section}
\counterwithin{equation}{section} 
\renewcommand\theequation{\thesection.\arabic{equation}} 

\section{Auxiliary Results} \label{appendixA}

\begin{proposition}[Global $L^2$-orthogonal projection onto $\RT_h$ and its convergence]
\label{prop:L2proj-RT} 
Let $\RT_h$ denote the Raviart--Thomas space of order $k$ on $\mathcal T_h$, conforming in $H(\divv;\Omega)$.

\smallskip
\noindent\textbf{(a) Construction (global $L^2$-orthogonal projector).}
For $u\in L^2(\Omega)^3$, define $P_h u\in \RT_h$ by
\begin{equation}\label{eq:L2proj-def}
  (u - P_h u,\, v_h) = 0 \qquad \forall\, v_h\in \RT_h.
\end{equation}
Equivalently, if $\{\rho_i\}_{i=1}^M$ is any basis of $\RT_h$, let
$M_{ij}:=(\rho_j,\rho_i)$ and $b_i:=(u,\rho_i)$ and solve $M c=b$; then
$P_h u:=\sum_{j=1}^M c_j\,\rho_j$.

\smallskip
\noindent Then $P_h:L^2(\Omega)^3\to \RT_h$ is a well-defined linear projector with
\[
P_h^2=P_h,\qquad \mathrm{Range}(P_h)=\RT_h,\qquad \|P_h u\|_{L^2(\Omega)}\le \|u\|_{L^2(\Omega)}
\]
\[
\text{and}\quad
\|u-P_h u\|_{L^2(\Omega)}=\min_{v_h\in\RT_h}\|u-v_h\|_{L^2(\Omega)}.
\]

\smallskip
\noindent\textbf{(b) Convergence for $L^2$-data.}
For every $\Phi\in L^2(\Omega)^3$,
\begin{equation}\label{eq:L2-conv}
  \|\,\Phi - P_h\Phi\,\|_{L^2(\Omega)} \xrightarrow[h\to0]{} 0 .
\end{equation}
\end{proposition}
\begin{proof}
See \cite[Prop.~3]{HAntil_2025a}.
\end{proof}
\begin{proposition}[Riesz projection on $H_0(\curl)$: stability and convergence]
\label{prop:Riesz-Hcurl} 
Define, for each $u\in H_0(\curl;\Omega)$, the element $R_h u\in \mathcal N_h^0$ by
\begin{equation}\label{eq:def-Rh}
  a(R_h u, v_h) \;=\; a(u, v_h)\qquad\forall\,v_h\in \mathcal N_h^0,
\end{equation}
where $a(u,v):=(u,v)+(\curl u,\curl v)$.
Then the following hold.

\smallskip\noindent{\bf (i) Well-posedness and linearity.}
For each $u$, there exists a unique $R_h u\in\mathcal N_h^0$ solving \eqref{eq:def-Rh}; the
map $R_h:H_0(\curl;\Omega)\to \mathcal N_h^0$ is linear.

\smallskip\noindent{\bf (ii) Projection and Galerkin orthogonality.}
$R_h$ is a projection: $R_h|_{\mathcal N_h^0}=\mathrm{Id}$.
Moreover,
\begin{equation}\label{eq:gal-orth}
  a(u - R_h u,\, v_h)=0\qquad \forall\,v_h\in \mathcal N_h^0.
\end{equation}

\smallskip\noindent{\bf (iii) Stability (contractivity) in $H(\curl)$.}
For all $u\in H_0(\curl)$,
\begin{equation}\label{eq:contract}
  \|R_h u\|_{H(\curl;\Omega)} \;\le\; \|u\|_{H(\curl;\Omega)},
\qquad
  \|u-R_h u\|_{H(\curl;\Omega)} \;=\; \min_{v_h\in\mathcal N_h^0}\|u-v_h\|_{H(\curl;\Omega)}.
\end{equation}

\smallskip\noindent{\bf (iv) Convergence (general).} 
For every $u\in H_0(\curl;\Omega)$,
\begin{equation}\label{eq:conv-general}
  \|u-R_h u\|_{H(\curl;\Omega)} \xrightarrow[h\to0]{} 0,
\quad\text{hence}\quad
  \|u-R_h u\|_{L^2(\Omega)} \to 0,\ \ \|\curl(u-R_h u)\|_{L^2(\Omega)}\to 0 .
\end{equation}
\end{proposition}
\begin{proof}
See \cite[Prop.~4]{HAntil_2025a} for a proof.
\end{proof}

\begin{lemma}[Time–cell averaging converges to the identity]
\label{lem:time-average-convergence}
Let $X$ be a Hilbert space (e.g.\ $X=L^2(\Omega)^3$). For a uniform partition
$\{I_n=(t^n,t^{n+1}]\}_{n=0}^{N_t-1}$ of $(0,T]$ with $\Delta t=t^{n+1}-t^n$,
define the piecewise–constant time–averaging operator
\[
  (\Pi_{\Delta t}\Phi)(s)
  := \frac1{\Delta t}\int_{I_n}\Phi(t)\,dt
  \qquad\text{for }s\in I_n,
\]
for $\Phi\in L^1(0,T;X)$. Then
\[
  \|\Pi_{\Delta t}\Phi-\Phi\|_{L^1(0,T;X)}\xrightarrow[\Delta t\to0]{}0.
\]
\end{lemma}

\begin{proof}
\emph{Step 1: $\Pi_{\Delta t}$ is $L^1$–contractive.}
For each $n$ and $s\in I_n$, by Jensen’s inequality in $X$,
\[
  \big\|(\Pi_{\Delta t}\Phi)(s)\big\|_X
  = \left\|\frac1{\Delta t}\int_{I_n}\Phi(t)\,dt\right\|_X
  \le \frac1{\Delta t}\int_{I_n}\|\Phi(t)\|_X\,dt .
\]
Integrating over $s\in I_n$ and summing in $n$ yields
$\|\Pi_{\Delta t}\Phi\|_{L^1(0,T;X)}\le \|\Phi\|_{L^1(0,T;X)}$.

\emph{Step 2: Convergence for smooth functions.}
If $\Phi\in C^0([0,T];X)$, let $\omega(\cdot)$ be its (scalar) modulus of continuity:
$\|\Phi(t)-\Phi(s)\|_X\le \omega(|t-s|)$ with $\omega(r)\to0$ as $r\downarrow0$.
For $s\in I_n$,
\[
  \big\|(\Pi_{\Delta t}\Phi)(s)-\Phi(s)\big\|_X
  \le \frac1{\Delta t}\int_{I_n}\|\Phi(t)-\Phi(s)\|_X\,dt
  \le \frac1{\Delta t}\int_{I_n}\omega(|t-s|)\,dt
  \le \omega(\Delta t).
\]
Hence $\|\Pi_{\Delta t}\Phi-\Phi\|_{L^1(0,T;X)}\le T\,\omega(\Delta t)\to0$.

\emph{Step 3: Density argument.}
$C^0([0,T];X)$ is dense in $L^1(0,T;X)$. Given
$\Phi\in L^1(0,T;X)$ and $\varepsilon>0$, select $\Phi_\varepsilon\in C^0([0,T];X)$ with
$\|\Phi-\Phi_\varepsilon\|_{L^1(0,T;X)}<\varepsilon$. Then, using Step~1 and Step~2,
\[
\begin{aligned}
\|\Pi_{\Delta t}\Phi-\Phi\|_{L^1(0,T;X)}
&\le \|\Pi_{\Delta t}(\Phi-\Phi_\varepsilon)\|_{L^1(0,T;X)}
   +\|\Pi_{\Delta t}\Phi_\varepsilon-\Phi_\varepsilon\|_{L^1(0,T;X)}
   +\|\Phi_\varepsilon-\Phi\|_{L^1(0,T;X)}\\
&\le 2\varepsilon+\|\Pi_{\Delta t}\Phi_\varepsilon-\Phi_\varepsilon\|_{L^1(0,T;X)}
  \xrightarrow[\Delta t\to0]{} 2\varepsilon.
\end{aligned}
\]
Since $\varepsilon>0$ is arbitrary, the claim follows.
\end{proof}

\begin{lemma}[Discrete IBP in time with a smooth cutoff]
\label{lem:disc-IBP}
Let $\{u^n\}_{n=0}^{N_t}\subset L^2(\Omega)^3$ and define the piecewise–linear lift
\[
u_h^{\mathrm{pl}}(t)
:= u^n + (t-t^n)\,\delta_t u^{n+1}
\qquad\text{for }t\in I_n:=(t^n,t^{n+1}],\quad
\delta_t u^{n+1}:=\frac{u^{n+1}-u^n}{\Delta t}.
\]
Fix $v_h\in L^2(\Omega)^3$ (constant in time) and $\eta\in C_c^\infty(0,T)$, and set the cell average
\[
\eta^{n+\frac12}:=\frac{1}{\Delta t}\int_{t^n}^{t^{n+1}}\eta(t)\,dt.
\]
Then
\begin{equation}\label{eq:disc-ibp}
\sum_{n=0}^{N_t-1}\Delta t\,(\delta_t u^{n+1},v_h)\,\eta^{n+\frac12}
\;=\;-\int_0^T (u_h^{\mathrm{pl}}(t),v_h)\,\eta'(t)\,dt .
\end{equation}
\end{lemma}

\begin{proof}
On each time cell $I_n=(t^n,t^{n+1}]$ the lift is affine, hence
$\partial_t u_h^{\mathrm{pl}}(t)=\delta_t u^{n+1}$ for $t\in I_n$.
Therefore
\[
\Delta t\,(\delta_t u^{n+1},v_h)\,\eta^{n+\frac12}
= (\delta_t u^{n+1},v_h)\int_{t^n}^{t^{n+1}}\eta(t)\,dt
= \int_{t^n}^{t^{n+1}}(\partial_t u_h^{\mathrm{pl}}(t),v_h)\,\eta(t)\,dt .
\]
Integrate by parts in time on $I_n$ (note that $u_h^{\mathrm{pl}}$ is continuous and piecewise $C^1$):
\[
\int_{t^n}^{t^{n+1}}(\partial_t u_h^{\mathrm{pl}},v_h)\,\eta
= \bigl( u_h^{\mathrm{pl}}(t),v_h\bigr)\,\eta(t)\Big|_{t^n}^{t^{n+1}}
  - \int_{t^n}^{t^{n+1}} (u_h^{\mathrm{pl}}(t),v_h)\,\eta'(t)\,dt .
\]
Summing over $n=0,\dots,N_t-1$ yields
\[
\sum_{n=0}^{N_t-1}\Delta t\,(\delta_t u^{n+1},v_h)\,\eta^{n+\frac12}
= \sum_{n=0}^{N_t-1}\bigl( u_h^{\mathrm{pl}}(t^{n+1}),v_h\bigr)\,\eta(t^{n+1})
 - \sum_{n=0}^{N_t-1}\bigl( u_h^{\mathrm{pl}}(t^n),v_h\bigr)\,\eta(t^n)
 - \int_0^T (u_h^{\mathrm{pl}}(t),v_h)\,\eta'(t)\,dt .
\]
The two discrete sums telescope. Since $\eta\in C_c^\infty(0,T)$, we have $\eta(0)=\eta(T)=0$, and because
$u_h^{\mathrm{pl}}$ is continuous at grid nodes (with $u_h^{\mathrm{pl}}(t^n)=u^n$), the boundary contribution vanishes:
\[
\sum_{n=0}^{N_t-1}\bigl( u_h^{\mathrm{pl}}(t^{n+1}),v_h\bigr)\,\eta(t^{n+1})
 - \sum_{n=0}^{N_t-1}\bigl( u_h^{\mathrm{pl}}(t^n),v_h\bigr)\,\eta(t^n)=0.
\]
This gives \eqref{eq:disc-ibp}.
\end{proof}

\bibliographystyle{plain}
\bibliography{HDJR}

\end{document}